\documentclass[a4paper,10pt]{article}

\usepackage[english]{babel} 

\usepackage[A4]{vmargin} 
\usepackage{multicol} 
\usepackage{graphicx} 
\usepackage{color, colortbl} 
\usepackage{lscape} 

\usepackage{bbm}
\usepackage{dsfont}
\usepackage{amsfonts}
\usepackage{amssymb}
\usepackage{mathrsfs}
\usepackage{amsthm}
\usepackage{amsmath}
\DeclareMathAlphabet{\mathpzc}{OT1}{pzc}{m}{it}
\usepackage{stmaryrd}
\usepackage{placeins}
\usepackage{hyperref}
\hypersetup{                    
    colorlinks=true,                
    breaklinks=true,                
    urlcolor= blue,                 
    linkcolor= red,                
    citecolor= blue                
    }

\newcommand{\vertiii}[1]{{\left\vert\kern-0.25ex\left\vert\kern-0.25ex\left\vert #1 
    \right\vert\kern-0.25ex\right\vert\kern-0.25ex\right\vert}}

\usepackage{verbatim} 
\usepackage{enumerate} 
\usepackage{eurosym} 
\usepackage{algorithm,algorithmic} 
\usepackage{pstricks,pst-node,pst-tree,pstricks-add}
\usepackage{color, colortbl}

\usepackage{tikz}

\usepackage[toc,page]{appendix}

\usepackage{makeidx}

\newtheorem{theorem}{Theorem}[section]

\theoremstyle{plain}

\theoremstyle{plain}
\newtheorem{remark}{Remark}[section]
\theoremstyle{plain}
\newtheorem{lemma}{Lemma}[section]
\theoremstyle{plain}

\theoremstyle{plain}

\setmarginsrb{3cm}{2cm}{3cm}{2cm}{2cm}{0cm}{0cm}{1cm}

\newcommand\pare[1]{\left(#1\right)}

\newcommand\croch[1]{\left[#1\right]}

\newcommand{\dx}{\partial_x}
\newcommand{\mbf}{\mathbf}
\newcommand{\dsp}{\displaystyle}
\newcommand{\R}{\mbf{R}}
\newcommand{\Z}{\mbf{Z}}
\newcommand{\T}{\mbf{T}}
\newcommand{\lo}{\mathcal{L}}
\newcommand{\xin}{\xi^{n+1}}
\newcommand{\roin}{\rho_i^n}
\newcommand{\roinn}{\rho_i^{n+1}}
\newcommand{\rojn}{\rho_j^n}
\newcommand{\rojnn}{\rho_j^{n+1}}
\newcommand{\gn}{g^n}
\newcommand{\gnn}{g^{n+1}}
\newcommand{\jn}{j^n}
\newcommand{\jnn}{j^{n+1}}
\newcommand{\Jn}{J^n}
\newcommand{\Jnn}{J^{n+1}}
\newcommand{\an}{\alpha^n}
\newcommand{\ann}{\alpha^{n+1}}
\newcommand{\bn}{\gamma^n}
\newcommand{\bnn}{\gamma^{n+1}}
\newcommand{\ip}{{i+\frac{1}{2}}}
\newcommand{\ipp}{{i+\frac{3}{2}}}
\newcommand{\im}{{i-\frac{1}{2}}}
\newcommand{\jp}{{j+\frac{1}{2}}}
\newcommand{\jpp}{{j+\frac{3}{2}}}
\newcommand{\jm}{{j-\frac{1}{2}}}
\newcommand{\Dt}{\Delta t}
\newcommand{\Dx}{\Delta x}

\newcommand{\E}{\mathbb{E}}
\newcommand{\fn}{\mathcal{F}_n}

\author{Nathalie Ayi\thanks{Laboratoire Jacques-Louis Lions, UPMC, Paris 6},  Erwan Faou\thanks{IRMAR Universit\'e de Rennes 1 \& INRIA}}

\title{Analysis of an asymptotic preserving scheme for stochastic linear kinetic equations in the diffusion limit}

\makeindex

\begin{document}
\maketitle

\textbf{Abstract.} We present an asymptotic preserving scheme based on a micro-macro decomposition for stochastic linear transport equations in kinetic and diffusive regimes. We perfom a mathematical analysis and prove that the scheme is uniformly stable with respect to the mean free path of the particles in the simple telegraph model and in the general case. We present several numerical tests which validate our scheme. \\

\textbf{Key words.} stochastic transport equations, diffusion limit, asymptotic preserving scheme, stiff terms, stability analysis

\section{Introduction}

In the physical contexts associated with neutron transport, radiative transfer, rarefied gas dynamics, the systems can be described at several scales: the microscopic one which is interested into the evolution of each particle, the macroscopic one which, as indicated by its name, deals with the macroscopic quantities. There exists also an intermediate scale called mesoscopic where, this time, the evolution of the density of particles satisfying a kinetic equation is studied. The change from one scale to another is done by passing to the limit on one parameter of the system: when starting at the mesoscopic scale, the passage to the limit is on the mean free path, denoted by $\varepsilon$, which goes to $0$.\\
\indent This article focuses on one specific type of limit: the diffusion one. From a theoretical point of view, this subject has been treated in various different frameworks. We can mention the passage from the BGK model to the Navier-Stokes equation \cite{SR}, from the Boltzmann equation to the incompressible Navier-Stokes equation \cite{GSR} or the convergence to the Rosseland approximation \cite{BGPS}. The  starting point of our motivation is a stochastic perturbation of this last case by a Wiener process as in \cite{DebDeMVov} (note that other types of stochastic version of this equation exist as in \cite{AyiGoud} but will not be treated here).\\
\indent Indeed,  lately, the study of stochastic perturbation of well known deterministic partial differential equations has been a subject of growing interest. The introduction of such term can be justified to model numerical and empirical uncertainties. What we are interested in here is a numerical study of these problems. \\
\indent These types of problems, associated with a change of scales, can be very challenging numerically. Because of the stiff terms which are contained in the kinetic equation, classical numerical methods are prohibitively expensive. What we would like is schemes which mimics the asymptotic behavior of the kinetic equation, i.e. reduce to numerical approximations of the macroscopic equation when the scaling parameter goes to $0$. This is exactly the purpose of the Asymptotic Preserving (AP) schemes. They have been first studied in neutron transport by Larsen, Morel and Miller \cite{LMM}, Larsen and Morel \cite{LaMo} and Jin and Levermore \cite{JinLev1,JinLev2} for steady problems. For time dependent problems, we can mention the works of Klar \cite{Klar}, Jin, Pareschi and Toscani \cite{JinPaTo} who proposed two classes of semi-implicit time discretizations. \\
\indent The starting point of this article is a scheme proposed by Lemou and Mieussens in \cite{LemouMieussens} based on the micro-macro decomposition of the distribution function into microscopic and macroscopic components. The decomposition only uses basic properties of the collision operator that are common to most of kinetic equations (namely conservation and equilibrium properties) and leads to a coupled system of equations for these two components without any linearity assumption. One of the interest of this approach is that it appears to be very general, as it can be applied to kinetic equations for both diffusion limit (see \cite{LemouMieussens,LiuMieussens} for linear transport equations and \cite{BLM_Kac} for the nonlinear Kac equation) and hydrodynamic regimes (see \cite{BLM_BGK} for the Boltzmann equation for instance).\\
\indent The aim of our article is to apply this method to obtain an AP scheme in the case of  linear kinetic equations with a stochastic perturbation modelled by a multiplicative Wiener process. To our knowledge, this is the first study of this type for stochastic kinetic equations with  multiplicative noise. Actually, though those equations are more and more studied from a theoretical point of view as mentioned previously, very little is done on that scope numerically, more precisely in the domain of AP schemes for stochastic partial differential equations. Still, note that there exists works of AP schemes in the presence of randomness in the context of uncertainty quantification (see  for instance \cite{JinLiu,JXZ,JLM}). The techniques developed in these latter cases are very different from the ones that we will adopt, that are linked with stochastic calculus. \\
\indent The paper is organized as follows : in Section \ref{setting}, we introduce the model under study, which is a stochastic kinetic linear equation with multiplicative noise, and we present its discretization by the AP scheme. Section \ref{telegraph} is devoted to the stability analysis in the simpler case of two discrete velocities, the telegraph equation, perturbed by a Brownian Motion. In Section \ref{generalcase}, we prove the stability in the general case under an explicit CFL condition. Finally, we present various numerical tests in Section \ref{Snumerics} which validates our scheme.

\section{General setting}
\label{setting}

We are interested into the following stochastic linear kinetic equation (see \cite{DebDeMVov}) 
\begin{equation}
\label{eq_diff_sto}
df + \frac{1}{\varepsilon} v \partial_x f dt = \frac{\sigma}{\varepsilon^2} \mathcal{L} f dt + f \circ Q dW_t
\end{equation}
where $f$ is the distribution function of particles that depends on time $t>0$, on position $x \in \T = \R / 2\pi\Z$ and on velocity $v \in [-1,1]$, $dW_t$ a cylindrical Wiener process on the Hilbert space $L^2(\T)$. We can define it by setting $$dW_t = \sum_{k\geq 0} e_k d\beta_k(t)$$ where the $(\beta_k)_{k \geq 0}$ are independent Brownian motions on the real line and $(e_k)_{k \geq 0}$ a complete orthonormal system in the Hilbert space  $L^2(\T)$. $Q$ is  a linear self-adjoint operator on $L^2(\T)$ such that 
\begin{equation}
\label{trace}
\sum_{k \geq 0} \| Q e_k\|_{L_x^\infty}^2 < + \infty.
 \end{equation} 
Moreover, we assume that $\sigma$ satisfies $0<\sigma_m \leq \sigma(x) \leq \sigma_M$ for every $x$.\\
 In Equation \eqref{eq_diff_sto}, the left-hand side represents the free transport of the particles while the right-hand side models the interaction of particles with the medium. \\
 
\noindent  We define the operator $\Pi$ such that 
 \begin{equation*}
 \Pi \phi = \frac{1}{2} \int_{-1}^1 \phi(v) dv
 \end{equation*}
 which is the average of every velocity dependent function  $\phi$. The linear operator $\lo$ that we will consider is given by 
 \begin{equation*}
 \lo f(v) = \int_{-1}^1 s(v,v') (f(v')-f(v)) dv' , 
 \end{equation*} 
 where the kernel $s$ is such that $0<s_m \leq s(v,v') \leq s_M$ for every $v,v' \in [-1,1]$. We assume that $s$ satisfies $\dsp{\int_{-1}^1 s(v,v') dv' = 1}$ and that it is symmetric: $s(v,v')=s(v',v)$. It is standard to state 
 the following properties :
\begin{center}
\begin{enumerate}
\item[$\bullet$] $\lo$ acts only on the velocity dependence of $f$ (it is local with respect to $t$ and $x$).
\item[$\bullet$] $\Pi (\lo \phi) = 0$ for every $\phi \in L^2([-1,1])$.
\item[$\bullet$] The null space of $\lo$ is $\mathcal{N}(\lo) = \{\phi = \Pi \phi\}$
 (constant functions).
 \item[$\bullet$] The rank of $\lo$ is $\mathcal{R}(\lo) = \mathcal{N}^\perp(\lo) = \{\phi \text{ s.t. } \Pi \phi = 0\}$.
  \item[$\bullet$] $\lo$ is non-positive self-adjoint in $ L^2([-1,1])$ and we have 
  \begin{equation}
  \label{bklp}
  \Pi (\phi \lo \phi) \leq -2 s_m \Pi (\phi^2)
  \end{equation}
  for every $\phi \in \mathcal{N}^\perp(\lo)$.
    \item[$\bullet$]  $\lo$ admits a pseudo inverse from $\mathcal{N}^\perp(\lo)$ onto $\mathcal{N}^\perp(\lo)$ denoted by $\lo^{-1}$.
         \item[$\bullet$] The orthogonal projection from $ L^2([-1,1])$ onto $\mathcal{N}^\perp(\lo)$ is $\Pi$.
\end{enumerate}
\end{center}
For instance, the one-group transport equation corresponds to 
$$
 \lo f  =  \dsp{\int_{-1}^1 \frac{1}{2} (f(v')-f(v))dv'}
   =  \Pi f -f, 
$$
and it is classical in this case to prove that $\lo$ satisfies all the previous properties. Equation \eqref{eq_diff_sto} becomes 
 \begin{equation*}
df + \frac{1}{\varepsilon} v \partial_x f dt = \frac{\sigma}{\varepsilon^2} (\Pi f -f) dt + f \circ Q dW_t.
\end{equation*}
If the velocity set is $\{-1,1\}$, $dv$ is the discrete Lebesgue measure and the corresponding one-group transport equation is called the telegraph equation. We denote $f(t,x,1) := p(t,x)$ and $f(t,x,-1) := q(t,x)$. For $\sigma = 1$, the equation \eqref{eq_diff_sto} becomes 
\begin{equation}
\label{galibier}
\left\{\begin{array}{l}
\dsp{dp + \frac{1}{\varepsilon} \dx p dt = \frac{1}{\varepsilon^2} (\frac{p+q}{2} - p) dt + p \circ dW_t}\\
\\
\dsp{dq - \frac{1}{\varepsilon} \dx q dt = \frac{1}{\varepsilon^2} (\frac{p+q}{2} - q) dt + q \circ dW_t}. 
\end{array}
\right.
\end{equation}

We want to construct an AP scheme associated with the diffusive limit of \eqref{eq_diff_sto} when $\varepsilon$ goes to $0$ which is 
\begin{equation}
\label{diff_eq}
d \rho - \partial_x \kappa \partial_x \rho dt = \rho \circ Q dW_t
\end{equation}
 with $\dsp{\kappa(x) = - \frac{\Pi ( v \lo^{-1} v)}{\sigma(x)}}$, see \cite{DebDeMVov}.

Quite similarly to the deterministic case in \cite{LemouMieussens,LiuMieussens}, we adopt a micro-macro decomposition. Indeed, we introduce $g$ such that  
\begin{equation}
\label{def_f_g}
f= \rho + \varepsilon g  ~\text{with } \rho = \Pi ( f ) ~\text{and } g ~\text{is such that } \Pi (g) = 0
\end{equation}
 and with the hypothesis on $\mathcal{L}$, we obtain an equivalent system to \eqref{eq_diff_sto}:
 \begin{equation}
 \label{micromacrostrato}
 \left\{\begin{array}{l}
 \dsp{d\rho + \partial_x \Pi (v g)  dt = \rho \circ Q dW_t} \\
 \dsp{dg + \frac{1}{\varepsilon} (I - \Pi)  (v \partial_x g) dt = \frac{\sigma}{\varepsilon^2} \lo g dt + g \circ Q dW_t - \frac{1}{\varepsilon^2} v \partial_x \rho dt}.
 \end{array} 
 \right.
 \end{equation}
 Using the formula which links the It\^o integral and the Stratanovich one, we can rewrite \eqref{micromacrostrato} as follows
  \begin{equation}
 \label{micromacro}
 \left\{\begin{array}{l}
 \dsp{d\rho + \partial_x \Pi (v g ) dt = \rho Q dW_t + \frac{1}{2} \rho \sum_{k \geq 0} (Q e_k)^2} dt \\
 \dsp{dg + \frac{1}{\varepsilon} (I - \Pi)  (v \partial_x g) dt = \frac{\sigma}{\varepsilon^2} \lo g dt + g  Q dW_t  + \frac{1}{2} g  \sum_{k \geq 0} (Q e_k)^2 dt - \frac{1}{\varepsilon^2} v \partial_x \rho dt}.
 \end{array} 
 \right.
 \end{equation}
 We study the following numerical scheme for this system with a time step $\Delta t$ and times $t_n = n \Delta t$ and two staggered grids of step $\Delta x$ and nodes $x_i = i \Delta x$ and $x_{i+\frac{1}{2}} = (i + \frac{1}{2}) \Delta x$ extended by periodicity. We are interested in a semi-discretization in $x$, and we use the notation $\rho_i^n \approx \rho(t_n,x_i)$ and $g_{i+\frac{1}{2}}^n (v) \approx g(t_n,x_{i+\frac{1}{2}},v)$.  \begin{subequations}
 \label{scheme}
\begin{gather} 
  \dsp{\roinn = \roin  - \Dt ~ \Pi \pare{ v \frac{\gnn_\ip - \gnn_\im}{\Dx}}  + \roin \pare{\frac{1}{2} \Delta t \sum_{k \geq 0} (b_{ik})^2  + \sqrt{\Delta t} \sum_{k \geq 0} b_{ik} \xi_k^{n+1}}}  ~~~~~~~~~~~~~~~~~~~  ~~~~~~~~~
  \label{scheme_a}\\
  \begin{array}{l}
  \dsp{\gnn_\ip = \gn_\ip - \frac{\Dt}{\varepsilon \Dx} (I - \Pi) \pare{v^+ \pare{\gn_\ip - \gn_\im} + v^- \pare{\gn_\ipp -\gn_\ip}}} \\
  \\
  ~~~~~~~~~~~~~~~~~~~~~~~~~\dsp{ - \frac{\sigma_\ip}{\varepsilon^2} \lo g^{n+1}_\ip \Dt + \gn_\ip \pare{\frac{1}{2} \Delta t \sum_{k \geq 0} (b_{i+\frac{1}{2},k})^2  + \sqrt{\Delta t} \sum_{k \geq 0} b_{i+\frac{1}{2},k} \xi_k^{n+1}}}
 \\
   ~~~~~~~~~~~~~~~~~~~~~~~~~~~~~~~~~~~~~~~~~~~~~~ \dsp{  - \frac{1}{\varepsilon^2} v \frac{\rho_{i+1}^n - \roin}{\Dx} \Dt},
 \end{array} \label{scheme_b}
 \end{gather} 
 \end{subequations} 
where $v^+=\max(v,0)$ and $v^-=\min(v,0)$, $(\xi_k^n)_{n \geq 1, k \geq 0}$ are i.i.d. variables with a normal distribution and  we use the notation  $b_{ik} := Qe_k(x_i)$ and  $b_{i + \frac{1}{2},k} := Qe_k(x_{i + \frac{1}{2}})$. \\
 
 Let us briefly comment the scheme \eqref{scheme}. Similarly to the deterministic case, we can observe that amongst the stiffest terms in $\varepsilon$, only the collision term is implicit. This will ensure stability as $\varepsilon$ goes to $0$. Furthermore, the upwind discretization of $(I- \Pi) (v \dx g)$ will ensure stability in the kinetic regime while the centered approximation of $\dx \Pi (v g )$ and $v \dx \rho$ will allow to capture the diffusion limit. Indeed, we have formally when $\varepsilon$ goes to $0$
 \begin{equation}
 \label{approxg}
 \gnn_\ip = - \frac{1}{\sigma_\ip} \lo^{-1}\pare{v \frac{\rho_{i+1}^n - \roin}{\Dx} \Dt} + O(\varepsilon).
 \end{equation}
 Therefore, using  \eqref{approxg} in \eqref{scheme_a}, we obtain when passing to the limit
 \begin{equation}
 \roinn = \roin - \frac{\Dt}{\Dx} \pare{\kappa_\ip \frac{\rho_{i+1}^n - \roin}{\Dx}  - \kappa_\im \frac{\rho_{i}^n - \rho_{i-1}^n}{\Dx}} + \roin \pare{\frac{1}{2} \Delta t \sum_{k \geq 0} (b_{ik})^2  + \sqrt{\Delta t} \sum_{k \geq 0} b_{ik} \xi_k^{n+1}}
 \end{equation}
 which is the usual 3-points stencil explicit scheme for the diffusion equation \eqref{diff_eq} with the notation $\dsp{\kappa_\ip = - \frac{\Pi ( v \lo^{-1} v)}{\sigma_\ip}}$.\\
 
 In the following, we are interested in the stability of this scheme. Of course, in our AP scheme context, we want to prove uniform stability with respect to $\varepsilon$. In the next section, we start with the simpler case of the telegraph equation in which we have only two discrete velocities $v=\pm1$, and a one dimensional Brownian motion. The general case is proved in section \ref{generalcase}.
 
 \section{The telegraph equation}
 \label{telegraph}

In the telegraph model introduced  previously, only two velocities $v = +1$ and $v = -1$ are present. As mentioned previously, in that case, the solution $f$ is thus determined by 
$f(t,x,1) := p(t,x)$ and   $f(t,x,-1) := q(t,x)$ and the equation \eqref{galibier} reads
 \begin{equation}
\label{telegraph_eq}
\left\{\begin{array}{l}
\dsp{dp + \frac{1}{\varepsilon} \dx p dt = \frac{1}{2 \varepsilon^2} ({q} - p) dt + p \circ d\beta(t)}\\
\\
\dsp{dq - \frac{1}{\varepsilon} \dx q dt = \frac{1}{2 \varepsilon^2} ({p} - q) dt + q \circ d\beta(t)}, 
\end{array}
\right.
\end{equation}
 with $\beta(t)$ a one dimensional Brownian motion. For the telegraph equation, $f$ is decomposed into $f=\rho E + \varepsilon g$, where $\rho = \frac{1}{2} (p+q) =: \Pi f$, $E=(1,1)$ and $g=(\alpha, \gamma)$ with $\Pi  g = 0$ which is written $\alpha=-\gamma$. Thus, the micro-macro system \eqref{micromacrostrato}  is written here
 \begin{equation}
\left\{\begin{array}{l}
\dsp{d \rho +  \dx \frac{\alpha - \gamma}{2} dt = \rho \circ d\beta(t)}\\
\\
\dsp{d \alpha + \frac{1}{\varepsilon} \dx \frac{\alpha + \gamma}{2} dt = - \frac{1}{\varepsilon^2} \alpha dt + \alpha \circ d\beta(t) - \frac{1}{\varepsilon^2} \dx \rho dt}\\
\\
\dsp{d \gamma -  \frac{1}{\varepsilon} \dx \frac{\gamma + \alpha}{2} dt = - \frac{1}{\varepsilon^2} \gamma dt + \gamma \circ d\beta(t) + \frac{1}{\varepsilon^2} \dx \rho dt}.\\
\end{array}
\right.
\end{equation}
For this system, the scheme \eqref{scheme}  takes the form  
 \begin{subequations}
 \begin{gather}
 \roinn = \roin - \frac{\Dt}{2 \Dx} \croch{\pare{\ann_\ip - \bnn_\ip} - \pare{\ann_\im - \bnn_\im}} + \roin \pare{\frac{\Dt}{2} + \sqrt{\Dt} ~\xin} \\
 \begin{array}{l}
 ~\\
 \dsp{\ann_\ip = \an_\ip {\color{red}-} \frac{\Dt}{\varepsilon \Dx} \croch{\pare{\an_\ip - \an_\im} - \frac{1}{2} \pare{\an_\ip - \an_\im - \bn_\ipp + \bn_\ip}}}\\
 ~~~~~~~~~~~~~~~~~~~~~~~~~~~~~~\dsp{- \frac{\Dt}{\varepsilon^2} \ann_\ip + \an_\ip \pare{\frac{\Dt}{2} + \sqrt{\Dt} ~\xin}  - \frac{1}{\varepsilon^2} \Dt \pare{\frac{\rho_{i+1}^n-\roin}{\Dx}}}
 \end{array}\\
  \begin{array}{l}
  ~\\
 \dsp{\bnn_\ip = \bn_\ip{\color{red}-} \frac{\Dt}{\varepsilon \Dx} \croch{-\pare{\bn_\ipp - \bn_\ip} - \frac{1}{2} \pare{\an_\ip - \an_\im - \bn_\ipp + \bn_\ip}}}\\
 ~~~~~~~~~~~~~~~~~~~~~~~~~~~~~~\dsp{- \frac{\Dt}{\varepsilon^2} \bnn_\ip + \bn_\ip \pare{\frac{\Dt}{2} + \sqrt{\Dt} ~\xin}  {\color{red}+}\frac{1}{\varepsilon^2} \Dt \pare{\frac{\rho_{i+1}^n-\roin}{\Dx}}}
 \end{array}
 \end{gather} 
  \end{subequations}
  where $(\xi^n)_{n \geq 1}$ are i.i.d. variables with a normal distribution.
  We denote $j := \frac{1}{2\varepsilon} (p-q) = \frac{1}{2} (\alpha - \gamma)$ and the above scheme can be written under the much simpler form
   \begin{subequations}
   \label{scheme_easy1}
 \begin{gather}
 \roinn = \roin - \frac{\Dt}{\Dx} \pare{\jnn_\ip  - \jnn_\im} + \roin \pare{\frac{\Dt}{2} + \sqrt{\Dt} ~\xin} ~~~~~~~~~~~~~~~~~~~~~~~~~~~~~~~~\\
 \begin{array}{l}
 \dsp{\jnn_\ip = \jn_\ip + \frac{\Dt}{2 \varepsilon \Dx} \croch{\jn_\ipp - 2 \jn_\ip + \jn_\im}}\\
 ~~~~~~~~~~~~~~~~~~~~~~~~~~~~~~\dsp{- \frac{\Dt}{\varepsilon^2} \jnn_\ip + \jn_\ip \pare{\frac{\Dt}{2} + \sqrt{\Dt} ~\xin}  - \frac{1}{\varepsilon^2} \Dt \pare{\frac{\rho_{i+1}^n-\roin}{\Dx}}}
 \end{array}
 \end{gather} 
  \end{subequations}
  In the following theorem, we prove the stability of this scheme. 
  \begin{theorem}
  \label{th_stability1}
  There exist constants $L$, $\Dt_0$, $\Dx_0$ and $\varepsilon_0$ such that for all $\Dt \leq \Dt_0$, $\Dx \leq \Dx_0$ and $\varepsilon \leq \varepsilon_0$ satisfying the CFL condition 
 \begin{equation}
 \label{lombric}
  \Dt \leq \frac{1}{2} \pare{\frac{\Dx^2}{2} + \varepsilon \Dx}
 \end{equation}
   then we have 
  \begin{equation*}
  \E \croch{\sum_i (\roin)^2 + (\varepsilon \jn_\ip)^2} \leq e^{L n \Dt} \, \E \croch{\sum_i (\rho_i^0)^2 + (\varepsilon j^0_\ip)^2} 
  \end{equation*}
  for every n.   \end{theorem}
  \begin{proof}
  Similarly to the deterministic case, see \cite{LemouMieussens}, the proof presented here is based on a standard Von Neumann analysis.  \\ 
  We introduce the following notations $\Jn_\ip = \varepsilon \jn_\ip$, $\mu = \frac{\Dt}{\varepsilon \Dx}$ and $\lambda = \frac{1}{1+ \Dt/\varepsilon^2}$. Then, \eqref{scheme_easy1} is rewritten as follows
     \begin{subequations}
   \label{scheme_easy}
 \begin{gather}
 \rojnn = \rojn - \mu \pare{\Jnn_\jp  - \Jnn_\jm} + \rojn \pare{\frac{\Dt}{2} + \sqrt{\Dt} ~\xin} ~~~~~~~~~~~~~~~~~~~~~~~~~~~~~~~~~~~~~~~\\
 \dsp{\Jnn_\jp =\lambda \left(  \Jn_\jp (1 + \frac{\Dt}{2} + \sqrt{\Dt} ~\xin)  + \frac{\mu}{2} \croch{\Jn_\jpp - 2 \Jn_\jp + \Jn_\jm} - \mu  \pare{\rho_{j+1}^n-\rojn} \right)}
 \end{gather} 
  \end{subequations}
  where the index $i$ has been repaced by $j$ to avoid confusion with $i = \sqrt{-1}$. We take $\rojn$ and $\Jn_\jp$ on the form of elementary waves $\rojn = \rho^n(\varphi) e^{ij\varphi}$ and $\Jn_\jp = \Jn(\varphi) e^{i (\jp) \varphi}$. As in \cite{LemouMieussens}, we are interested into finding a relation between the amplitudes and conclude by linearity of the scheme. We obtain the following one 
  \begin{equation}
  \left\{
  \begin{array}{l}
  \dsp{\rho^{n+1} = \rho^n  (1 + \frac{\Dt}{2}  + \sqrt{\Dt} ~\xin) - 2 i \mu J^{n+1} \sin \theta} \\[2ex]
  \dsp{J^{n+1} = \lambda \pare{J^n  \pare{1 + \frac{\Dt}{2} + \sqrt{\Dt} ~\xin - 2 \mu \sin^2\theta} - 2i\mu \sin \theta \rho^n}}
  \end{array}
  \right.
  \end{equation}
  with $\dsp{\theta = \frac{\varphi}{2}}$. Under a matrix form, this rewrite as 
  \begin{equation}
  \begin{pmatrix}
  \rho^{n+1} \\ 
  J^{n+1}
  \end{pmatrix} = A_{n+1} \begin{pmatrix}
  \rho^{n} \\ 
  J^{n}
  \end{pmatrix}
  \end{equation}
  with $$ A_{n+1} = \begin{pmatrix}
  \dsp{1+ \frac{\Dt}{2}  + \sqrt{\Dt} ~\xin  - 4 \mu^2 \lambda \sin^2 \theta}  &\dsp{ -i (  1+ \frac{\Dt}{2}  + \sqrt{\Dt} ~\xin  - 2 \mu  \sin^2 \theta) 2 \lambda \mu \sin \theta}  \\ 
\dsp{  -2 i \mu \lambda\sin \theta  }&  \dsp{(  1+ \frac{\Dt}{2}  + \sqrt{\Dt} ~\xin  - 2 \mu  \sin^2 \theta) \lambda}.
  \end{pmatrix} $$
  At this point, we notice that $A_{n+1}$ is a stochastic perturbation of the matrix $\widetilde{A}$ appearing in \cite{LemouMieussens} : 
  \begin{equation}
  \label{atilde} \widetilde{A} = \begin{pmatrix}
  \dsp{1  - 4 \mu^2 \lambda \sin^2 \theta}  &\dsp{ -i (  1  - 2 \mu  \sin^2 \theta) 2 \lambda \mu \sin \theta}  \\ 
\dsp{  -2 i \mu \lambda \sin \theta  }&  \dsp{(  1  - 2 \mu  \sin^2 \theta) \lambda}
  \end{pmatrix}. 
  \end{equation}

Indeed, we can write 
\begin{equation}
\label{an1}
 A_{n+1} = \widetilde{A} + \sqrt{\Dt} \xi^{n+1} B + \Dt \,  C
\end{equation}
where the matrices $B = B(\mu \lambda,\mu^2 \lambda, \sin \theta)$ and $C = C(\mu \lambda,\mu^2 \lambda, \sin \theta)$ are explicitly given from the expression of $A_{n+1}$. 
%
%
%
Now, note that  under the assumption \eqref{lombric}, we have
$$
\lambda \mu = \frac{\Dt}{\varepsilon \Dx( 1 + \Dt /\varepsilon^2)} \leq \min( \frac{\varepsilon }{\Dx}, \frac{\Dt}{\varepsilon \Dx}) \leq  \min( \frac{\varepsilon }{\Dx}, \frac{1}{2}( 1 + \frac{\Dx}{2\varepsilon}))  \leq C
$$ 
for some fixed constant $C$, and similarly 
$$
\lambda \mu^2 = \frac{\Dt^2}{\varepsilon^2 \Dx^2( 1 + \Dt /\varepsilon^2)}\leq \min ( \frac{\Dt^2}{\varepsilon^2 \Dx^2}, \frac{\Dt}{ \Dx^2}) \leq \min ( \frac{1}{4}( 1 + \frac{\Dx}{2\varepsilon})^2, \frac12 (1 + \frac{\varepsilon}{\Dx})) \leq C. 
$$
Hence the quantities $\lambda \mu$ and $\lambda \mu^2$ are uniformly bounded under the CFL condition \eqref{lombric}. We deduce that the matrices $B(\lambda \mu, \lambda\mu^2,\sin \theta)$  and $C(\lambda \mu, \lambda\mu^2,\sin \theta)$ are uniformly bounded with respect to $\theta$, and $\Dt$, $\Dx$ and $\varepsilon$ satisfying the condition of the Theorem.\\
Let us denote by $\mathcal{F}_n$ the $\sigma$-algebra generated by $\rho^n, J^n,\xi^n, \rho^{n-1}, J^{n-1}, \xi^{n-1}, \dots, \rho^0,J^0$, then,  by construction $\xi^{n+1}$ is independent of  $\mathcal{F}_n$ and $\rho^n, J^n$ are $\mathcal{F}_n$-measurable. Therefore, by properties of the conditional expectation, we have by explicit calculations using the fact that $\E (\xi^{n+1} |  \mathcal{F}_n)= 0$, 
\begin{multline*}
 \dsp{ \E \croch{\left\| \begin{pmatrix}
  \rho^{n+1} \\ 
  J^{n+1}
  \end{pmatrix} \right\|_2^2\left|  \mathcal{F}_n \right.} } =  \dsp{ \E \croch{\left\|  A_{n+1}\begin{pmatrix}
  \rho^{n} \\ 
  J^{n}
  \end{pmatrix} \right\|_2^2\left|  \mathcal{F}_n \right.} }\\
  \leq 
  \dsp{ \left\|  (\widetilde{A} + \Dt C) \begin{pmatrix}
  \rho^{n} \\ 
  J^{n}
  \end{pmatrix} \right\|_2^2} + \Dt \dsp{ \left\| B \begin{pmatrix}
  \rho^{n} \\ 
  J^{n}
  \end{pmatrix} \right\|_2^2} 
\leq  \dsp{ \Big(\|  \widetilde A \|_2^2 +  L \Dt \Big)  \left\| \begin{pmatrix}
  \rho^{n} \\ 
  J^{n}
  \end{pmatrix} \right\|_2^2  },
 \end{multline*}
 for $\Dt$ small enough, 
  where the constant $L$ depends on bounds on the matrix $\widetilde{A},B$ and $C$. \\
Now we will prove that under the condition \eqref{lombric}, we have $\|  \widetilde A \|_2^2 \leq 1$ which shows the result by induction. Indeed, $\|  \widetilde A \|_2^2$ is the largest eigenvalue of the matrix $\widetilde{A}^* \widetilde{A}$. Denoting by $\widetilde{T}$ and $\widetilde{D}$ the trace an determinant of this latter matrix, the largest eigenvalue is 
$$
\frac{\widetilde{T} + \sqrt{\widetilde{T}^2 - 4 \widetilde{D}}}{2}
$$
and the condition $\|  \widetilde A \|_2^2 \leq 1$ is thus equivalent to $1 - \widetilde{T} + \widetilde{D} \geq 0$. 
Now we calculate that, with $X = (\sin \theta)^2$, $\widetilde D = \lambda^2 (1 - 2 \mu X)^2$ and 
$$
\widetilde{T} = 4 \lambda^2 \mu^2 X + \lambda^2 (1 - 2 \mu X)^2 + 4 \lambda^2 \mu^2 X(1 - 2\mu X)^2 + (1 - 4 \lambda \mu^2 X)^2. 
$$
Hence 
\begin{eqnarray*}
1 -  \widetilde{T} +  \widetilde{D} &=&  1 -   4 \lambda^2 \mu^2 X - 4 \lambda^2 \mu^2 X(1 - 2\mu X)^2 - (1 - 4 \lambda \mu^2 X)^2\\
&=&   -   4 \lambda^2 \mu^2 X - 4 \lambda^2 \mu^2 X(1 - 2\mu X)^2 + 8 \lambda \mu^2 X - 16 \lambda^2 \mu^4 X^2  \\
&=& 4 \lambda \mu^2 X ( - \lambda -  \lambda ( 1 - 2 \mu X)^2 + 2 - 4 \lambda \mu^2 X) \\
&=& 8 \lambda \mu^2 X ( 1 - \lambda + 2\lambda \mu X  - 2 \lambda \mu^2 X^2 - 2 \lambda \mu^2 X). 
\end{eqnarray*}
The stability of the deterministic case is thus ensured if this expression is non negative for all $X \in [0,1]$. Now the polynomial $Q(X) := 1 - \lambda + 2\lambda \mu X  - 2 \lambda \mu^2 X^2 - 2 \lambda \mu^2 X$ is concave and satisfies $Q(0) = 1 - \lambda > 0$. Hence the condition will be satisfied if $Q(1) \geq 0$ which is written 
$$
1 - \lambda + 2\lambda \mu   - 4 \lambda \mu^2   \geq 0
$$
And we easily verify that this condition is ensured under the CFL condition \eqref{lombric}. This finishes the proof of the Theorem.
   \end{proof}

 \section{Stability analysis for the stochastic linear kinetic equations}
 \label{generalcase}
 
 We go back to the study of the general case and establish the uniform stability of the scheme \eqref{scheme}. 
 \begin{theorem}
 If $\Dt$ satisfies the following CFL condition 
 \begin{equation}
 \label{CFL}
 \Dt \leq \frac{2 s_m \sigma_m \Dx^2}{2(2+\varepsilon)} + \frac{\varepsilon \Dx}{2+\varepsilon},
 \end{equation}
  then the sequence $\rho^n$ and $g^n$ defined by the scheme \eqref{scheme} satisfy the energy estimate
 \begin{equation}
   \E \croch{\sum_i (\roin)^2} + \varepsilon^2 \E \croch{\sum_i  \Pi  \pare{(\gn_\ip)^2}} \leq {C(T)} \pare{\E \croch{\sum_i (\rho_i^0)^2} + \varepsilon^2 \E \croch{\sum_i  \Pi  \pare{(g^0_\ip)^2}}}
\end{equation}  for every $n$ with ${C(T)}$ a constant which only depends on $T$. Hence, the scheme \eqref{scheme} is stable.
 \end{theorem}
 
 \subsection{Notations and basic properties}
 
 We adopt the same notations as in \cite{LemouMieussens}. We denote by $M$ the number of points of the grid associated with the discrete positions $x_i$ and $J:=\{1,  \dots, M\}$. For every grid function $\mu = (\mu_i)_{i \in J}$, we define
 \begin{equation}
 \| \mu \|^2 =  \sum_{i } \mu_i^2 \Dx.
 \end{equation}
 For every velocity dependent grid function $v \in [-1,1] \mapsto \phi(v) = \pare{\phi_\ip(v)}_{i \in J}$, we define 
 \begin{equation}
\vertiii{ \phi }^2 =  \sum_{i } \Pi  \pare{ \phi_\ip^2}  \Dx.
 \end{equation}
 If $\phi$ and $\psi$ are  two velocity dependent grid functions, we define their inner product
 \begin{equation}
 \langle \phi,\psi \rangle = \sum_i \Pi  \pare{\phi_\ip \psi_\ip} \Dx. 
 \end{equation}
 We also give some notations for the finite difference operators which are used in scheme \eqref{scheme}. For every grid function $\phi =(\phi_\ip)_{i \in J}$, we define the following one-sided operators:
 \begin{equation}
 D^- \phi_\ip = \frac{\phi_\ip - \phi_\im}{\Dx} ~\text{and }~ D^+ \phi_\ip = \frac{\phi_\ipp - \phi_\ip}{\Dx} 
\end{equation}  
and the following centered operators:
  \begin{equation}
 D^c \phi_\ip = \frac{\phi_\ipp - \phi_\im}{\Dx} ~\text{and }~ D^0 \phi_i = \frac{\phi_\ip - \phi_\im}{\Dx} \pare{= D^- \phi_\ip}.
\end{equation}  
Finally, for every grid function $\mu = (\mu_i)_{i \in J}$, we define the following centered operator:
\begin{equation}
\delta^0 \mu_\ip = \frac{\mu_{i+1} - \mu_i}{\Dx}.
\end{equation}
Let us recall here some results about these operators whose proofs can be found in \cite{LiuMieussens}.
\begin{lemma}
For every grid function $\phi =(\phi_\ip)_{i \in J}$, $\psi =(\psi_\ip)_{i \in J}$ and   $\mu = (\mu_i)_{i \in J}$, we have
\begin{enumerate}
\item[$\bullet$] Centered form of the upwind operator:
\begin{equation}
\label{centerd_form}
(v^+D^- + v^-D^+) \phi_\ip = v D^c \phi_\ip - \frac{\Dx}{2} |v| D^-D^+ \phi_\ip
\end{equation}
\item[$\bullet$]  A priori bound for the discrete derivative:  \begin{equation} 
{\sum_i (D^+ \phi_\ip)^2 \Dx \leq \frac{4}{\Dx^2} \sum_i \phi_\ip^2 \Dx}
\end{equation}
\item[$\bullet$] Discrete integration by parts: 
\begin{equation}
\label{int_part1}
\sum_i \mu_i D^0 \phi_i \Dx = - \sum_i (\delta^0 \mu_\ip) \phi_\ip \Dx ~~~~~~~~~~ 
\end{equation}
\begin{equation}
\label{int_part2}
\sum_i \psi_\ip D^- \phi_\ip \Dx = - \sum_i (D^+ \psi_\ip) \phi_\ip \Dx
\end{equation}
\begin{equation}
\label{int_part3}
\sum_i \phi_\ip D^c \phi_\ip \Dx = 0~~~~~~~~~~~~~~~  ~~~~~~~~~~~~~
\end{equation}
\item[$\bullet$] Estimate for the adjoint upwind operator : for every positive real number $\alpha$ and for $\phi$ and $\psi$ being velocity dependent
\begin{equation}
\label{ineg_upwind}
|\langle (v^+D^+ + v^- D^-)\psi, \phi \rangle| \leq \alpha \vertiii{ \phi } + \frac{1}{4\alpha}  \vertiii{ \, |v| D^+ \psi \, }^2.
\end{equation}
\end{enumerate}
\end{lemma}
Finally, the operator $\Pi$ satisfies also the following property:
\begin{lemma}
If $g \in L^2([-1,1])$, then
\begin{equation}
\pare{\Pi  \pare{ v g}}^2 \leq \frac{1}{2} \Pi  \pare{|v| g^2}.
\end{equation}
\end{lemma}

\subsection{Energy estimates}
Using the notations introduced in the previous section, the scheme \eqref{scheme} can be written as 
 \begin{subequations}
   \label{scheme2}
\begin{gather} 
  \dsp{\roinn = - \Dt D^0 \Pi  \pare{v \gnn_i }+ \roin \pare{1 + \frac{1}{2} \Delta t \sum_{k \geq 0} (b_{ik})^2  + \sqrt{\Delta t} \sum_{k \geq 0} b_{ik} \xi_k^{n+1}} } ~~~~~~~~~~~~~~~~~~~~~~~~~~  ~~~~~~~~~
  \label{scheme2_a}\\
  \begin{array}{l}
  \dsp{\gnn_\ip = - \frac{\Dt}{\varepsilon \Dx} (I - \Pi) \pare{v^+ D^- + v^- D^+} \gn_\ip} \\
  \\
  ~~~~~~~~~~~~~~~~~~~~~~~~~\dsp{ - \frac{\sigma_\ip}{\varepsilon^2} \lo g^{n+1}_\ip \Dt + \gn_\ip \pare{1+\frac{1}{2} \Delta t \sum_{k \geq 0} (b_{i+\frac{1}{2},k})^2  + \sqrt{\Delta t} \sum_{k \geq 0} b_{i+\frac{1}{2},k} \xi_k^{n+1}} }\\
\dsp{  ~~~~~~~~~~~~~~~~~~~~~~~~~  ~~~~~~~~~~~~~~~~~~~~~~~~~  - \frac{1}{\varepsilon^2} v  \delta^0 \rho_\ip^n \Dt}.
 \end{array} \label{scheme2_b}
 \end{gather} 
 \end{subequations} 
 The energy of the system \eqref{micromacrostrato} being defined as $\dsp{\int \rho^2 dx + \varepsilon^2 \int \Pi  \pare{ g^2 } dx}$, similarly to the telegraph equation case, it is clear that the scheme can be proved to be stable if the discrete energy at time $n+1$ can be controlled by the discrete  energy at time $n$.
\\
 
 Therefore, we multiply \eqref{scheme2_a} by $\roinn$ and we take the sum over $i$. Thus, using the standard equality $a (a-b) = \frac{1}{2}(a^2 - b^2 + |a-b|^2)$, we obtain
 
 \begin{multline}
 \label{eq_rho_no_esp}
 \frac{1}{2} \left(\| \rho^{n+1}\|^2 - \left\| \pare{1 + \frac{1}{2} \Delta t \sum_{k \geq 0} (b_{\bullet k})^2  + \sqrt{\Delta t} \sum_{k \geq 0} b_{ \bullet k} \xi_k^{n+1}} \rho^n\right\|^2 \right. \\  \left.+ \left\| \rho^{n+1} - \pare{1 + \frac{1}{2} \Delta t \sum_{k \geq 0} (b_{\bullet k})^2  + \sqrt{\Delta t} \sum_{k \geq 0} b_{ \bullet k} \xi_k^{n+1}} \rho^n\right\|^2 \right) \\
+  \sum_i \roinn D^0 \Pi  \pare{ v g_i^{n+1}} \Dx \Dt =0
 \end{multline}
 denoting $\dsp{\sum_k b_{\bullet k} := (\sum_k b_{ik})_{i\in J}}$.
 Similarly to the proof of Theorem \ref{th_stability1}, we  want to take the conditional expectation $\E \croch{~\cdot~ | \mathcal{F}_n}$. We applied it on \eqref{eq_rho_no_esp}. The second term of the left-hand side becomes  
\begin{multline*}
\dsp{\E \croch{\left\| \pare{1+ \frac{1}{2} \Delta t \sum_{k \geq 0} (b_{\bullet k})^2  + \sqrt{\Delta t} \sum_{k \geq 0} b_{ \bullet k} \xi_k^{n+1}} \rho^n\right\|^2 | \fn}} \\ 
= \dsp{\E \left[\sum_i \pare{1+ \frac{1}{2} \Delta t \sum_{k \geq 0} b_{ik}^2  + \sqrt{\Delta t} \sum_{k \geq 0} b_{ik} \xi_k^{n+1}}^2 \pare{\roin}^2 \Dx | \fn \right]}
\end{multline*}
which is equal to 
\begin{multline*}
   \dsp{\E \left[\sum_i  \left(1 + \frac{1}{4} \Dt^2 \pare{\sum_{k \geq 0} b_{ik}^2}^2 + \Dt  \pare{\sum_{k \geq 0} b_{ik}}^2 (\xin_k)^2 + \Dt \pare{\sum_{k \geq 0} b_{ik}^2}\right. \right.} \\
\dsp{\left. \left.   + 2 \sqrt{\Dt} \pare{\sum_{k \geq 0} b_{ik} \xin_k} + \Dt^{3/2} \pare{\sum_{k \geq 0} b_{ik}^2} \pare{\sum_{k \geq 0} b_{ik} \xin_k} \right) \pare{\roin}^2 \Dx | \fn \right]}. 
 \end{multline*} 
This term can be written
\begin{multline*}
     \dsp{\sum_i  \pare{\roin}^2 \Dx  \E \left[\left(1 + \frac{1}{4} \Dt^2 \pare{\sum_{k \geq 0} b_{ik}^2}^2 + \Dt  \pare{\sum_{k \geq 0} b_{ik}}^2 (\xin_k)^2 \right.  \right.} \\
\dsp{\left.  \left. + \Dt \pare{\sum_{k \geq 0} b_{ik}^2}  + 2 \sqrt{\Dt} \pare{\sum_{k \geq 0} b_{ik} \xin_k} + \Dt^{3/2} \pare{\sum_{k \geq 0} b_{ik}^2} \pare{\sum_{k \geq 0} b_{ik} \xin_k} \right)  | \fn \right] }
\end{multline*}
which yields
\begin{multline*}
      \dsp{\sum_i  \pare{\roin}^2 \Dx  \pare{1 + \frac{1}{4} \Dt^2 \pare{\sum_{k \geq 0} b_{ik}^2}^2 + \Dt  \pare{\sum_{k \geq 0} b_{ik}}^2    + \Dt \pare{\sum_{k \geq 0} b_{ik}^2}}}\\
=  \dsp{\| \rho^n\|^2 + \Dt  \pare{\left\| \pare{\sum_{k \geq 0} b_{\bullet k}} \rho^n\right\|^2 + \left\| \pare{\sqrt{\sum_{k \geq 0} b_{\bullet k}^2}} \rho^n\right\|^2} + \frac{\Dt^2}{4} \left\| \pare{\sum_{k \geq 0} b_{\bullet k}^2} \rho^n\right\|^2}
\end{multline*} 
using the fact that $\rho^n$ is $\fn$-measurable and for all $k$, $\xin_k$ is independent of $\fn$ and the properties of the conditional expectation. Furthermore,  similarly, the third term of the left-hand side of \eqref{eq_rho_no_esp} becomes
\begin{multline*}
\dsp{\E \croch{\left\| \rho^{n+1} - \pare{1 + \frac{1}{2} \Delta t \sum_{k \geq 0} (b_{\bullet k})^2  + \sqrt{\Delta t} \sum_{k \geq 0} b_{ \bullet k} \xi_k^{n+1}} \rho^n\right\|^2 | \fn}}\\
\dsp{= \E \croch{\sum_i \pare{\roinn -\roin - \pare{\frac{1}{2} \Delta t \sum_{k \geq 0} b_{ik}^2  + \sqrt{\Delta t} \sum_{k \geq 0} b_{ik} \xi_k^{n+1}} \roin }^2 \Dx | \fn}}
\end{multline*}
\begin{multline*}
\dsp{= \E \left[\sum_i \left(\pare{\roinn -\roin}^2 + \pare{\frac{1}{2} \Delta t \sum_{k \geq 0} b_{ik}^2  + \sqrt{\Delta t} \sum_{k \geq 0} b_{ik} \xi_k^{n+1}}^2  (\roin)^2  \right. \right.}\\
\dsp{\left. - 2 \pare{\roinn -\roin}  \pare{\frac{1}{2} \Delta t \sum_{k \geq 0} b_{ik}^2  + \sqrt{\Delta t} \sum_{k \geq 0} b_{ik} \xi_k^{n+1}}   \roin  \Dx | \fn\right]}
\end{multline*}
This term can be written 
\begin{equation*}
\begin{array}{l}
\dsp{ =\E \croch{\left\| \rho^{n+1} -\rho^n\right\|^2 | \fn} } \\
~~~~\dsp{+  \E \croch{ \sum_i \pare{ \frac{1}{4} \Dt^2 \pare{\sum_{k \geq 0} b_{ik}^2 }^2 + \Dt \pare{\sum_{k \geq 0} b_{ik}}^2~(\xin_k)^2 + \Dt^{3/2} \pare{ \sum_{k \geq 0} b_{ik}^2 } \pare{\sum_{k \geq 0} b_{ ik} \xi_k^{n+1}} }  (\roin)^2 \Dx | \fn }}\\
~~~~~~~~~~~~~~~~~~~~~~~  \dsp{ - 2 \E \croch{ \sum_i  \pare{\roinn -\roin} \pare{\frac{1}{2} \Delta t \sum_{k \geq 0} b_{ik}^2  + \sqrt{\Delta t} \sum_{k \geq 0} b_{ik} \xi_k^{n+1}} \roin \Dx  | \fn}}\\
\end{array}
\end{equation*} 
\begin{equation*}
\begin{array}{l}
\dsp{ =\E \croch{\left\| \rho^{n+1} -\rho^n\right\|^2 | \fn} + \Dt   \left\| \pare{\sum_{k \geq 0} b_{\bullet k}} \rho^n\right\|^2  + \frac{\Dt^2}{4} \left\| \pare{\sum_{k \geq 0} b_{\bullet k}^2} \rho^n\right\|^2  } \\
~~~~~~~~~~~~~~~~~~~~~~~  \dsp{ - 2 \E \croch{ \sum_i  \pare{\roinn -\roin} \pare{\frac{1}{2} \Delta t \sum_{k \geq 0} b_{ik}^2  + \sqrt{\Delta t} \sum_{k \geq 0} b_{ik} \xi_k^{n+1}}  \roin \Dx  | \fn}}. 
\end{array}
\end{equation*} 
Thus, we have 
  \begin{multline}
     \label{expr_rho}
 \frac{1}{2} \pare{ \E \croch{\| \rho^{n+1}\|^2 | \fn}  -  \pare{\| \rho^n\|^2 + \Dt \left\| \pare{\sqrt{\sum_{k \geq 0} b_{\bullet k}^2}} \rho^n\right\|^2} + \E \croch{\left\| \rho^{n+1} -\rho^n\right\|^2 | \fn}}    \\
+  \E \croch{\sum_i \roinn D^0 \Pi  \pare{ v g_i^{n+1} } \Dx \Dt | \fn} \\
 -  \E \croch{ \sum_i  \pare{\roinn -\roin} \pare{\frac{1}{2} \Delta t \sum_{k \geq 0} b_{ik}^2  + \sqrt{\Delta t} \sum_{k \geq 0} b_{ik} \xi_k^{n+1}}  \roin \Dx  | \fn} =0.
 \end{multline}
 We do the same for $\gn$ multiplying \eqref{scheme2_b} by $\gn_\ip$, taking the velocity average and summing over $i$, we obtain
  \begin{multline*}
 \frac{1}{2} \left( \vertiii{ g^{n+1}}^2 - \vertiii{ \pare{1 +  \frac{1}{2} \Delta t \sum_{k \geq 0} (b_{\bullet k})^2  + \sqrt{\Delta t} \sum_{k \geq 0} b_{ \bullet k} \xi_k^{n+1}} g^n}^2  \right.\\ \left. + \vertiii{  g^{n+1} - \pare{1 +  \frac{1}{2} \Delta t \sum_{k \geq 0} (b_{\bullet k})^2  + \sqrt{\Delta t} \sum_{k \geq 0} b_{ \bullet k} \xi_k^{n+1}} g^n}^2\right) \\
+ \frac{\Dt}{\varepsilon} \langle \gnn, (I- \Pi) (v^+D^-+v^-D^+) \gn\rangle \\
= \frac{\Dt}{\varepsilon^2} \langle \gnn, \sigma \lo \gnn \rangle - \frac{\Dt}{\varepsilon^2} \sum_i \Pi  \pare{ v \gnn_\ip }\delta^0 \rho^n_\ip \Dx.
 \end{multline*}
 
 Again, we take the conditional expectation and we obtain an expression similar as the one obtained for $\rho^n$, 
   \begin{multline}
   \label{expr_g}
 \frac{1}{2} \pare{ \E \croch{\vertiii{ g^{n+1}}^2 | \fn}  -  \pare{\vertiii{ g^n}^2 + \Dt \vertiii{ \pare{\sqrt{\sum_{k \geq 0} b_{\bullet k}^2}} g^n}^2} + \E \croch{\vertiii{g^{n+1} - g^n}^2 | \fn }} \\
  -  \E \croch{ \sum_i  \Pi  \pare{ \pare{\gnn_\ip -\gn_\ip} \pare{\frac{1}{2} \Delta t \sum_{k \geq 0} b_{i+\frac{1}{2},k}^2  + \sqrt{\Delta t} \sum_{k \geq 0} b_{i+\frac{1}{2}, k} \xi_k^{n+1}}  \gn_\ip } \Dx  | \fn}\\
  +\E \croch{ \frac{\Dt}{\varepsilon}  \langle\gnn, (I- \Pi) (v^+D^-+v^-D^+) \gn  \rangle | \fn} \\
=  \E \croch{\frac{\Dt}{\varepsilon^2}  \langle\gnn, \sigma \lo \gnn  \rangle | \fn }-  \E \croch{\frac{\Dt}{\varepsilon^2} \sum_i \Pi  \pare{ v \gnn_\ip }\delta^0 \rho^n_\ip \Dx| \fn }.
 \end{multline}
 First, we notice that the fifth term of the left-and side of \eqref{expr_g} can be rewritten 
 \begin{equation*}
 \begin{array}{l}
   \dsp{\E \croch{ \frac{\Dt}{\varepsilon}  \langle \gnn, (I- \Pi) (v^+D^-+v^-D^+) \gn  \rangle | \fn} }\\
   \dsp{=   \E \croch{ \frac{\Dt}{\varepsilon}  \langle\gnn,  (v^+D^-+v^-D^+) \gn  \rangle | \fn} -   \E \croch{ \frac{\Dt}{\varepsilon} \sum_i \Pi  \pare{\gnn_\ip } \Pi  \pare{ (v^+D^-+v^-D^+) \gn_\ip } | \fn}}\\
    \dsp{=   \E \croch{ \frac{\Dt}{\varepsilon}  \langle \gnn,  (v^+D^-+v^-D^+) \gn  \rangle | \fn} -   \E \croch{ \frac{\Dt}{\varepsilon} \sum_i \Pi  \pare{\gnn_\ip }   | \fn}  \Pi  \pare{ (v^+D^-+v^-D^+) \gn_\ip }}
 \end{array}
 \end{equation*}
 Since the initial data satisfy $\Pi  \pare{g_\ip^0}= 0$ for every $i$ (see \eqref{def_f_g}), we can prove by induction that for all $n$,  $\mathbb{P}$-a.s. we have $\Pi  \pare{g_\ip^{n+1}}= 0$. Indeed, applying the average operator $\Pi$ to \eqref{scheme2_b} and using that $\Pi  \pare{  I -  \Pi }= 0$, $\Pi  \lo  =0$ and $\Pi  \pare{v }=0$ yields
 \begin{equation}
   \label{rec_gn}
\Pi  \pare{ g_\ip^{n+1}} = \pare{1 +\frac{1}{2} \Delta t \sum_{k \geq 0} b_{i+\frac{1}{2},k}^2  + \sqrt{\Delta t} \sum_{k \geq 0} b_{i+\frac{1}{2}, k} \xi_k^{n+1}} \Pi  \pare{g_\ip^{n}}.
 \end{equation}
 Therefore, 
the fifth term of the left-hand side of \eqref{expr_g} becomes 
  \begin{equation}
  \label{ineg_ps}
 \begin{array}{l}
    \dsp{  \E \croch{ \frac{\Dt}{\varepsilon}  \langle\gnn,  (v^+D^-+v^-D^+) \gn  \rangle | \fn} }
 \end{array}
 \end{equation}
 Furthermore, using the assumptions on $\sigma$ and the operator $\lo$ and the properties of the conditional expectation, we have
 \begin{equation}
 \label{ineg_lo}
 \E \croch{ \frac{\Dt}{\varepsilon^2}   \langle \gnn, \sigma \lo \gnn  \rangle   | \fn} \leq - \underbrace{2 s_m \sigma_m}_{=: \widetilde{\sigma}} \E \croch{\|| \gnn\||^2 | \fn} \Dt. 
 \end{equation}
 Thus, we add up \eqref{expr_rho}  and $\varepsilon^2$ $\times$ \eqref{expr_g} and we use \eqref{ineg_ps} and \eqref{ineg_lo} and the discrete integration by parts \eqref{int_part1}. We obtain 
  \begin{multline}
  \label{ineg_inter}
 \frac{1}{2} \pare{ \E \croch{\| \rho^{n+1}\|^2 | \fn}  -  \pare{\| \rho^n\|^2 + \Dt \left\| \pare{\sqrt{\sum_{k \geq 0} b_{\bullet k}^2}} \rho^n\right\|^2}}    +  \E \croch{\sum_i \roinn D^0 \Pi  \pare{ v g_i^{n+1} }\Dx \Dt | \fn} \\
 +  \frac{\varepsilon^2}{2} \pare{ \E \croch{\vertiii{ g^{n+1}}^2 | \fn}  -  \pare{\vertiii{ g^n}^2 + \Dt \vertiii{ \pare{\sqrt{\sum_{k \geq 0} b_{\bullet k}^2}} g^n}^2}} +{\varepsilon} \E \croch{ {\Dt}  \langle \gnn, (v^+D^-+v^-D^+) \gn  \rangle | \fn} \\
 \leq  \E \croch{ \sum_i  \pare{\roinn -\roin} \pare{\frac{1}{2} \Delta t \sum_{k \geq 0} b_{ik}^2  + \sqrt{\Delta t} \sum_{k \geq 0} b_{ik} \xi_k^{n+1}}  \roin \Dx  | \fn} ~~~~~~~~~~~~~~~~~~~~~~~~ ~~~~~~~~~~ \\ 
 + \varepsilon^2  \E \croch{ \sum_i  \Pi  \pare{ \pare{\gnn_\ip -\gn_\ip} \pare{ \frac{1}{2} \Delta t \sum_{k \geq 0} b_{i+\frac{1}{2},k}^2  + \sqrt{\Delta t} \sum_{k \geq 0} b_{i+\frac{1}{2}, k} \xi_k^{n+1}} \gn_\ip } \Dx  | \fn} \\
-\widetilde{\sigma} \E \croch{\vertiii{ \gnn}^2 | \fn} \Dt + \E \croch{\Dt \sum_i \Pi  \pare{v D^0 g_i^{n+1} }\roin \Dx  |\fn}.
 \end{multline}
 In the following, we will eliminate $\rho^{n+1}-\rho^{n}$  in \eqref{ineg_inter}.  Noticing that in \eqref{ineg_inter} the term 
 $$ 
 \E \croch{\sum_i \roinn D^0 \Pi  \pare{ v g_i^{n+1} } \Dx \Dt | \fn}
 $$ 
 in the left-hand side can be rewritten as $ \dsp{\E \croch{\sum_i  \Pi  \pare{ v D^0  g_i^{n+1} }  \roinn\Dx \Dt | \fn}}$, it can be coupled with the last term of the right-hand side. We thus want to control the term    $$ \E \croch{\sum_i \Pi  \pare{v D^0  g_i^{n+1} } (\roin - \roinn)\Dx \Dt | \fn}.$$ Using the Young inequality and the properties of the conditional expectation, we get for all $\alpha > 0$, 
 \begin{multline*}
 \left| \E \croch{\sum_i  \Pi  \pare{v D^0  g_i^{n+1} }  (\roin - \roinn)\Dx \Dt | \fn}  \right| \\
 \leq \alpha \E \croch{\| \rho^{n+1} -  \rho^n\|^2 \Dt | \fn} + \frac{1}{4 \alpha} \E \croch{\sum_i  \pare{\Pi  \pare{v D^0  g_i^{n+1} }}^2 \Dx \Dt | \fn} 
 \end{multline*}
 Quite similarly, we have
 \begin{equation*}
 \begin{array}{l}
\left| \dsp{\E \croch{ \sum_i  \pare{\roinn -\roin} \pare{\frac{1}{2} \Delta t \sum_{k \geq 0} b_{ik}^2  + \sqrt{\Delta t} \sum_{k \geq 0} b_{ik} \xi_k^{n+1}}  \roin \Dx  | \fn}} \right|\\
\dsp{\leq \alpha   \E \croch{\| \rho^{n+1} -  \rho^n\|^2 \Dt | \fn} + \frac{1}{4 \alpha}  \E \croch{ \left\| \pare{\frac{1}{2} \sqrt{\Delta t} \sum_{k \geq 0} b_{\bullet k}^2  + \sum_{k \geq 0} b_{\bullet k} \xi_k^{n+1}}  \rho^n \right\|^2  | \fn}}\\
 = \dsp{\alpha   \E \croch{\| \rho^{n+1} -  \rho^n\|^2 \Dt | \fn} + \frac{1}{4 \alpha} \pare{ \frac{1}{4} {\Dt} \left\| \sum_{k \geq 0} b_{\bullet k}^2 \rho^n\right\|^2 + \left\| \sum_{k \geq 0} b_{\bullet k} \rho^n\right\|^2}.   }
 \end{array}
 \end{equation*}
 Thus, $\rho^{n+1}-\rho^{n}$ terms cancel out in \eqref{ineg_inter} if $\alpha = \frac{1}{4 \Dt}$  and \eqref{ineg_inter} becomes 
  \begin{equation}
  \begin{array}{l}
  \label{ineg_inter2}
\dsp{ \frac{1}{2}   \E \croch{\| \rho^{n+1}\|^2 | \fn}}\\
 \dsp{ -\frac{1}{2}   \left( \| \rho^n\|^2 + \Dt \pare{\left\| \pare{\sqrt{\sum_{k \geq 0} b_{\bullet k}^2}} \rho^n\right\|^2  + 2 \left\| \pare{\sum_{k \geq 0} b_{\bullet k}} \rho^n\right\|^2}  + \frac{\Dt^2}{2} \left\| \pare{\sum_{k \geq 0} b_{\bullet k}^2} \rho^n\right\|^2 \right) }
  \\
\dsp{ +  \frac{\varepsilon^2}{2} \pare{ \E \croch{\|| g^{n+1}\||^2 | \fn}  -  \pare{\vertiii{ g^n}^2 + \Dt \vertiii{ \pare{\sqrt{\sum_{k \geq 0} b_{\bullet k}^2}} g^n}^2}}}\\
~\\
\dsp{ +{\varepsilon} \E \croch{ {\Dt}  \langle \gnn, (v^+D^-+v^-D^+) \gn  \rangle | \fn}} \\
~\\

 \leq   \varepsilon^2  \E \croch{ \sum_i  \Pi  \pare{ \pare{\gnn_\ip -\gn_\ip} \pare{ \frac{1}{2} \Delta t \sum_{k \geq 0} b_{i+\frac{1}{2},k}^2  + \sqrt{\Delta t} \sum_{k \geq 0} b_{i+\frac{1}{2}, k} \xi_k^{n+1}} \gn_\ip } \Dx  | \fn} \\
-\widetilde{\sigma} \E \croch{\vertiii{ \gnn}^2 | \fn} \Dt + \E \croch{\Dt \sum_i \Pi  \pare{v D^0 g_i^{n+1} } \roin \Dx  |\fn}.
  
  \end{array}
 \end{equation}
 Let us now prove that quite similarly, $\gnn - \gn$ can be eliminated. As in the deterministic case, we insert $\gnn$ in the inner product appearing in the left-hand side of \eqref{ineg_inter2} and we obtain 
 \begin{equation*}
 \begin{array}{l}
 \dsp{\E \croch{   \langle \gnn, (v^+D^-+v^-D^+) \gn  \rangle | \fn}}\\
  \dsp{ = \E \croch{   \langle\gnn, (v^+D^-+v^-D^+) \gnn \rangle | \fn} +\E \croch{   \langle \gnn, (v^+D^-+v^-D^+) (\gn-\gnn)  \rangle | \fn}} \\
  =: A+B
 \end{array}
 \end{equation*}
 Thus, using the centered form of the upwind operator \eqref{centerd_form} and the discrete integration by parts \eqref{int_part2} and \eqref{int_part3} , $A$ can be written 
 \begin{equation*}
 \begin{array}{rcl}
 A & =  & \dsp{\E \croch{ \langle \gnn, v D^c  \gnn  \rangle | \fn } -  \E \croch{ \langle \gnn, |v| D^- D^+  \gnn  \rangle | \fn } \frac{\Dx}{2}}\\
  & = & \dsp{\frac{\Dx}{2} \E \croch{   \langle D^+ \gnn, |v| D^+  \gnn  \rangle | \fn }}\\
  & = & \dsp{\frac{\Dx}{2} \E \croch{\sum_i \Pi  \pare{ |v| (D^+ \gnn)^2 } \Dx | \fn }}. 
 \end{array}
 \end{equation*}
 Using the discrete integration by parts \eqref{int_part2}, we obtain
 \begin{equation*}
 B = - \E \croch{   \langle(v^+D^+ +v^-D^-) \gnn, \gn-\gnn  \rangle | \fn}
 \end{equation*}
 and 
 \begin{equation*}
|B| \leq  \alpha \E \croch{  \| \gnn- \gn\|^2 | \fn} + \frac{1}{4\alpha} \E \croch{\vertiii{ \,  |v| D^+ \gnn\, }^2 | \fn}
 \end{equation*}
using \eqref{ineg_upwind}. Furthermore, using again the Young inequality, we obtain
\begin{equation*}
\begin{array}{l}
\dsp{ \E \croch{ \sum_i  \Pi  \pare{ \pare{\gnn_\ip -\gn_\ip} \pare{ \frac{1}{2} \Delta t \sum_{k \geq 0} b_{i+\frac{1}{2},k}^2  + \sqrt{\Delta t} \sum_{k \geq 0} b_{i+\frac{1}{2}, k} \xi_k^{n+1}} \gn_\ip } \Dx  | \fn}} \\
 \dsp{\leq \alpha \E \croch{\|| g^{n+1} -g^n\||^2 \Dt | \fn}+ \frac{1}{4\alpha} \pare{ \frac{1}{4} {\Dt} \vertiii{ \sum_{k \geq 0} b_{\bullet k}^2 g^n}^2 + \vertiii{ \sum_{k \geq 0} b_{\bullet k} g^n} ^2} }
\end{array}
\end{equation*}
Thus, for $\dsp{\alpha = \frac{\varepsilon}{2 \Dt (1+ \varepsilon)}}$, \eqref{ineg_inter2} becomes 
  \begin{equation}
  \label{ineg_inter3}
  \begin{array}{l}
 \dsp{\frac{1}{2}  \E \croch{\| \rho^{n+1}\|^2 | \fn}} \\
 \dsp{  -  \frac{1}{2} \pare{\| \rho^n\|^2 + \Dt \pare{\left\| \pare{\sqrt{\sum_{k \geq 0} b_{\bullet k}^2}} \rho^n\right\|^2  + 2 \left\| \pare{\sum_{k \geq 0} b_{\bullet k}} \rho^n\right\|^2} + \frac{\Dt^2}{2} \left\| \pare{\sum_{k \geq 0} b_{\bullet k}^2} \rho^n\right\|^2}  } 
  \\
 \dsp{+  \frac{\varepsilon^2}{2} \E \croch{\vertiii{ g^{n+1}}^2 | \fn}   }
 \\
 \dsp{-  \frac{\varepsilon^2}{2}  \pare{\vertiii{ g^n}^2 + \Dt \vertiii{ \pare{\sqrt{\sum_{k \geq 0} b_{\bullet k}^2}} g^n}^2  + \frac{(1+\varepsilon)}{\varepsilon} \pare{\Dt \vertiii{ \pare{\sum_{k \geq 0} b_{\bullet k}} g^n}^2 + \frac{\Dt^2}{4} \vertiii{ \pare{\sum_{k \geq 0} b_{\bullet k}^2} g^n}^2}}}\\
 \dsp{ +{\varepsilon} \frac{\Dx}{2} \E \croch{\sum_i \Pi  \pare{|v| (D^+ \gnn)^2 } \Dx | \fn } -  \frac{\Dt^2 (1+\varepsilon)}{2}  \E \croch{ \vertiii{\, |v| D^+ \gnn\,}^2 | \fn}} \\
 \leq  
-\widetilde{\sigma} \E \croch{\vertiii{ \gnn}^2 | \fn} \Dt + \E \croch{\Dt \sum_i \Pi  \pare{ v D^0 g_i^{n+1} } \roin \Dx  |\fn}.
  
  \end{array}
 \end{equation}
Similarly to the deterministic case, we can prove that $D^+ \gnn$ and $D^0 \gnn$ are controlled by $\vertiii{ \gnn }$ and we obtain  
 \begin{equation*}
  \begin{array}{l}
 \dsp{\frac{1}{2} \E \croch{\| \rho^{n+1}\|^2 | \fn}}\\
 \dsp{  - \frac{1}{2}   \pare{\| \rho^n\|^2 + \Dt \pare{\left\| \pare{\sqrt{\sum_{k \geq 0} b_{\bullet k}^2}} \rho^n\right\|^2  + 2 \left\| \pare{\sum_{k \geq 0} b_{\bullet k}} \rho^n\right\|^2} + \frac{\Dt^2}{2} \left\| \pare{\sum_{k \geq 0} b_{\bullet k}^2} \rho^n\right\|^2}  } 
  \\
 \dsp{+  \frac{\varepsilon^2}{2} \E \croch{\vertiii{ g^{n+1}}^2 | \fn}   }
 \\
 \dsp{ -   \frac{\varepsilon^2}{2} \pare{\vertiii{ g^n}^2 + \Dt \vertiii{| \pare{\sqrt{\sum_{k \geq 0} b_{\bullet k}^2}} g^n}^2  + \frac{(1+\varepsilon)}{\varepsilon} \pare{\Dt \vertiii{ \pare{\sum_{k \geq 0} b_{\bullet k}} g^n}^2 + \frac{\Dt^2}{4} \vertiii{ \pare{\sum_{k \geq 0} b_{\bullet k}^2} g^n}^2}}} \\
 \leq  
 \dsp{\Dt \pare{\pare{\frac{\Dt (1+\varepsilon)}{2} + \frac{\Dt}{2} - \frac{\varepsilon \Dx}{2}}^+  \frac{4}{\Dx^2} -  \widetilde{\sigma}}\E \croch{\vertiii{ \gnn}^2 | \fn}},
  \end{array}
 \end{equation*}
  see \cite{LiuMieussens} for more details. 
 Thus, using the assumption \eqref{trace}, we see that if $\Dt$ is such that 
\begin{equation}
\label{condition1}
 \pare{\frac{\Dt (1+\varepsilon)}{2} + \frac{\Dt}{2} - \frac{\varepsilon \Dx}{2}}^+  \frac{4}{\Dx^2} \leq  \widetilde{\sigma},
\end{equation} 
we finally obtain the energy estimate
 \begin{multline*}
\E \croch{\| \rho^{n+1}\|^2 | \fn}    
 +  \varepsilon^2 \E \croch{\vertiii{ g^{n+1}}^2 | \fn}  \\
\leq  (1 +  C_1\Dt + C_2 \Dt^2) \| \rho^n\|^2   +  \varepsilon^2 (1 +  \tilde{C}_1\Dt + \tilde{C}_2 \Dt^2) \vertiii{g^n}^2
 \end{multline*}
 for some constants 
$C_1,C_2,\tilde{C}_1,\tilde{C}_2$ depending on the parameter $b_{ik}$ defining the noise. 
 Taking the expectation, and assumming that $\Dt\leq 1$ yields 
$$
  \label{estimation_energy}
\E \croch{\| \rho^{n+1}\|^2}    
 +  \varepsilon^2 \E \croch{\vertiii{ g^{n+1}}^2  }  
\leq  C^{\Dt}   ( \E\croch{ \| \rho^n\|^2}   +  \varepsilon^2  \E \croch{\vertiii{ g^n}^2}),
$$
for some constant $C$, 
 and we obtain the result by induction. 
 Note that a sufficient condition ensuring \eqref{condition1} is 
 \begin{equation}
 \Dt \leq \frac{\widetilde{\sigma} \Dx^2}{2(2+ \varepsilon)} + \frac{\varepsilon \Dx}{2+\varepsilon}
 \end{equation}
 and \eqref{estimation_energy} concludes the proof in the case of an interval of finite time. 
 
 \begin{remark}
 We notice that unlike the deterministic version, the stability is proved here only for finite time interval.
  \end{remark}
 
 \section{Numerical tests\label{Snumerics}}
 We now compare the results obtained with our stochastic AP micro-macro scheme \eqref{scheme} (referred to as SMM) to the results obtained with a standard explicit discretization of the original equation \eqref{eq_diff_sto} written in kinetic variable (for $\varepsilon = 1$) and to the results obtained with the Crank-Nicholson scheme (referred to as CN) for the diffusion limit (for $\varepsilon=10^{-2}$) for the one-group transport equation
  \begin{equation}
d f + \frac{v}{\varepsilon} \partial_x f dt = \frac{1}{\varepsilon} (\Pi f - f) dt  + f \circ Q dW_t.
 \end{equation}
The space domain is $[0,1]$ discretized with $N=200$ points and we use periodic boundary conditions. The initial data is $$ f_0(x,v) = (1 + \cos(2\pi x+\pi)).$$ The time steps for the different scheme are chosen according to each associated CFL condition. We recall that for our scheme, it corresponds to \eqref{CFL}.\\

The noise is taken under the form $$ d\beta_0 + \sum_{k=1}^{N/2} \frac{1}{k+1} \pare{\cos(kx) + \sin(kx)} d \beta_k + \sum^{-1}_{k=-N/2}  \frac{1}{N/2 - (k+N/2-1)} \pare{\cos(kx) + \sin(kx)} d \beta_k$$
where the  $(\beta_k)_{k \in \{-N/2, \dots, N/2\}}$ are independent Brownian motions on the real line. \\

For both the diffusion and the kinetic case, for each time, we perform $100$ realizations. In all the following, the curves are the mean of $\rho$ over the $100$ realizations. \\

We start by showing in Figure \ref{c1} and \ref{c2} a comparison of the mean of $\rho$  for $\varepsilon=1$ between our SMM scheme and the explicit one at $t=0.1$, $t=0.3$, $t=0.6$ and $t=1$.  Figure \ref{c} shows the evolution of the mean of $\rho$ in the same picture. \\
We acknowledge that our scheme describes quite well the solution at any time. \\

In Figure \ref{d1} and \ref{d2}, we compare the mean of $\rho$ in our scheme $\varepsilon =10^{-2}$ with the one in a Crank-Nicholson scheme for the diffustion equation  at $t=\varepsilon/10$,  $t=4 \varepsilon/10$,  $t=0.05$ and $t=0.1$. We give the same conclusion as in the kinetic case above.

 \bibliographystyle{plain}
\bibliography{APStoch}

 \newpage
 \label{numerics}
 \begin{figure}
 \begin{center}
  \includegraphics[scale=0.57]{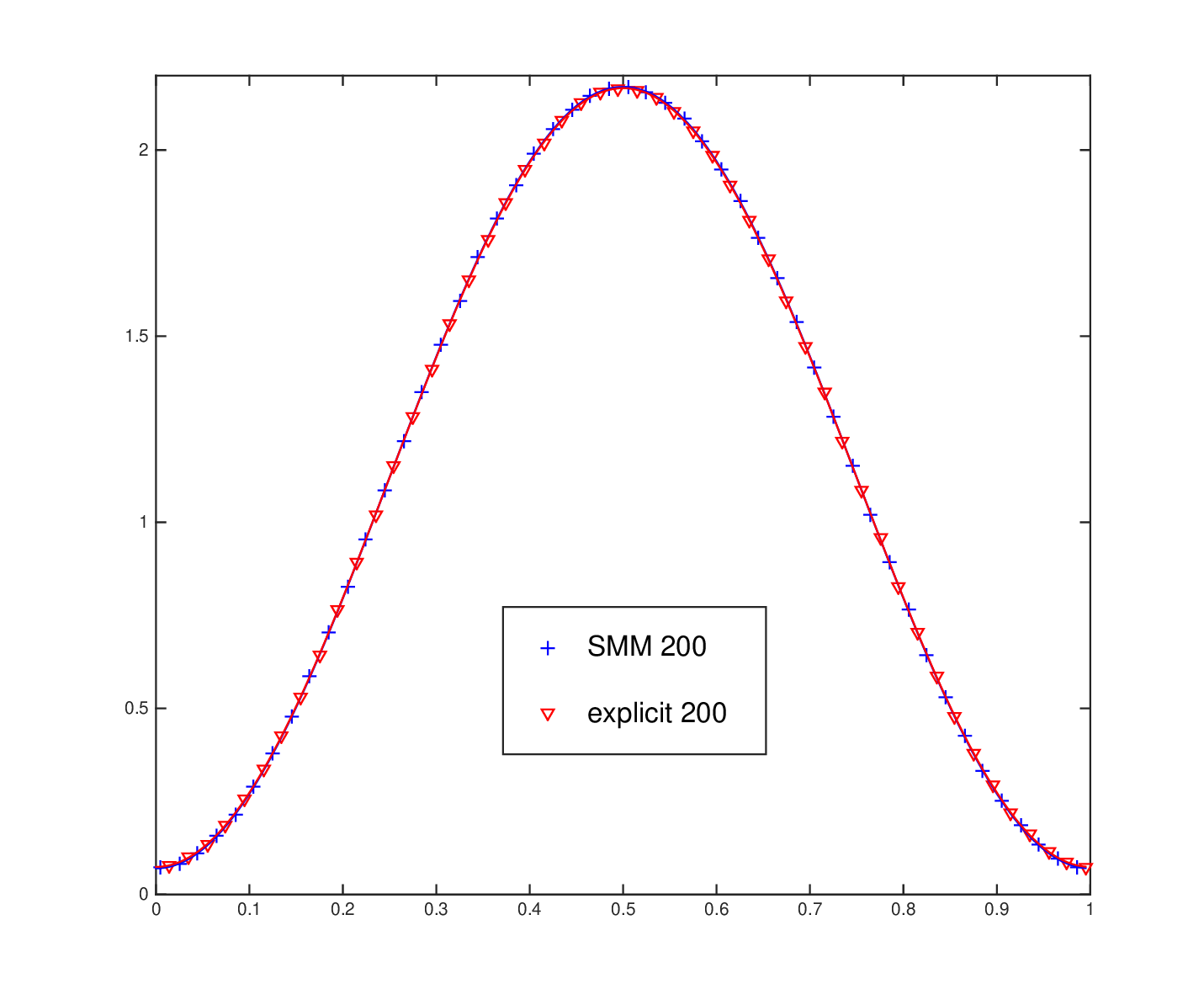}
  \includegraphics[scale=0.57]{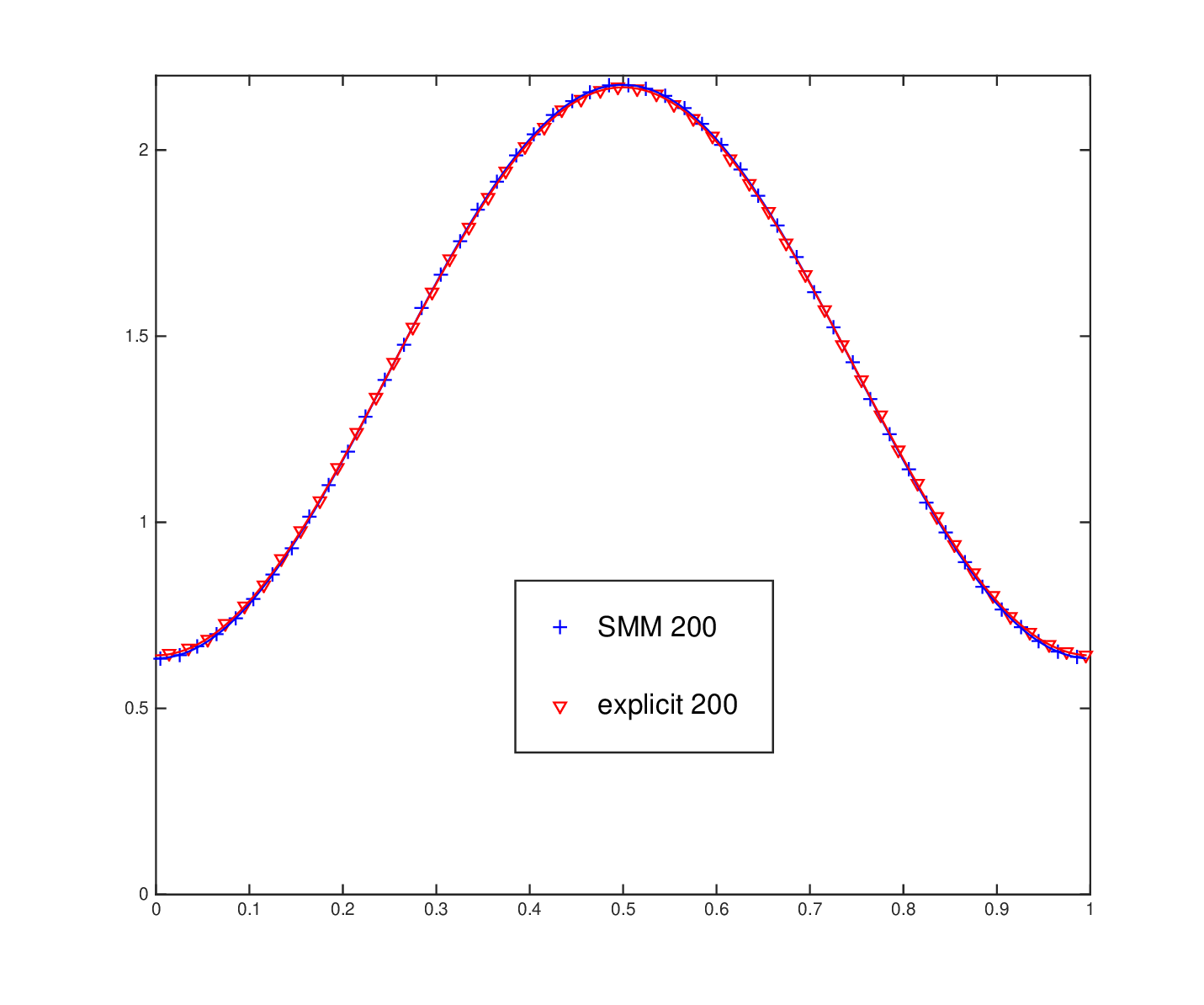}
 \end{center}
  \caption{one-group transport equation $\varepsilon = 1$ : comparison between SMM and explicit schemes (200 grid points): $t = 0.1$ (top), $t = 0.3$ (bottom).\label{c1}}
 \end{figure}
 
  \newpage
 \label{numerics}
 \begin{figure}
 \begin{center}
 \includegraphics[scale=0.57]{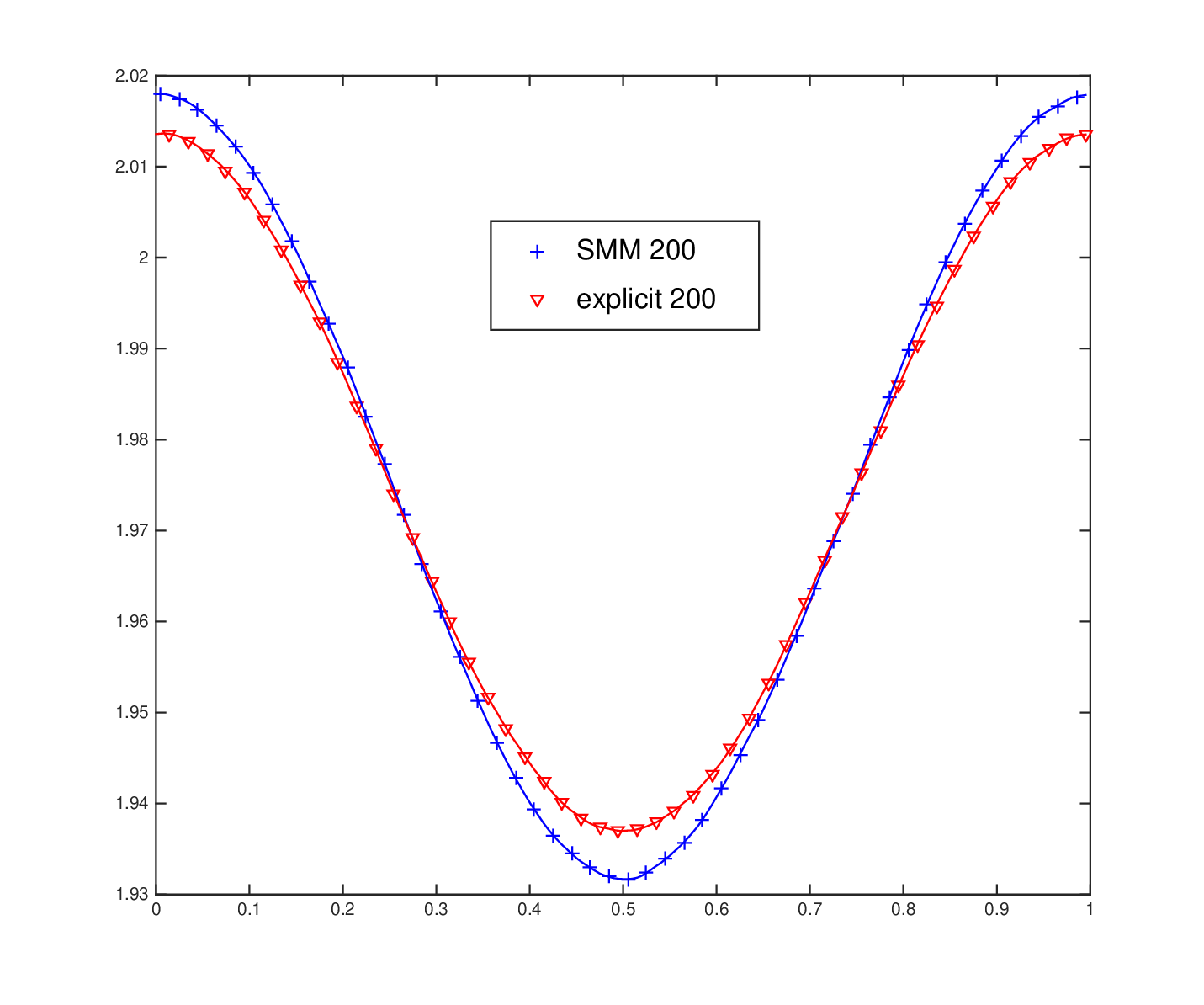}
  \includegraphics[scale=0.57]{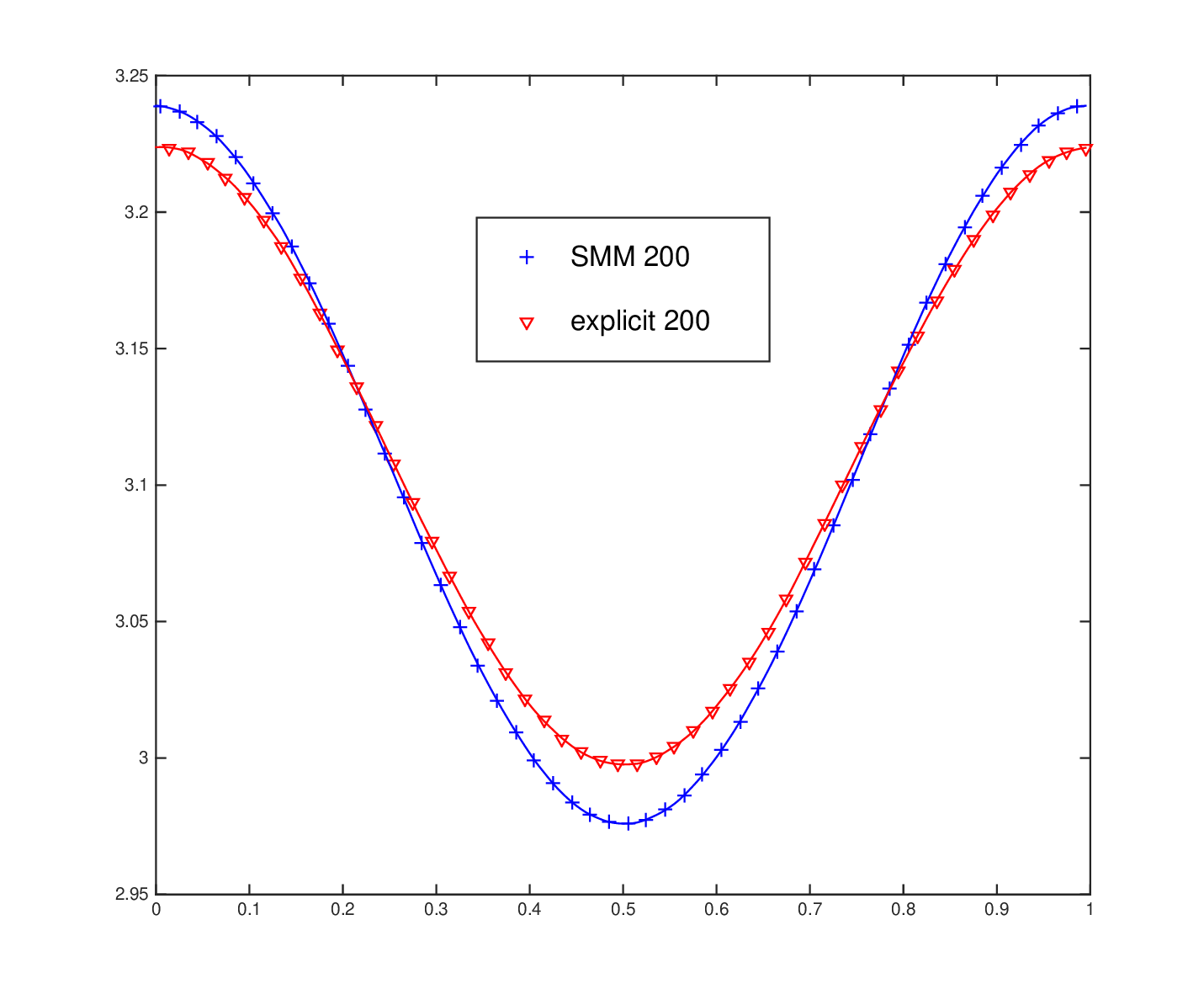}
 \end{center}
  \caption{one-group transport equation $\varepsilon = 1$ : comparison between SMM and explicit schemes (200 grid points): $t = 0.6$ (top), $t = 1$ (bottom). \label{c2}}
 \end{figure}

  \newpage
 \label{numerics}
 \begin{figure}
 \begin{center}
 \includegraphics[scale=0.57]{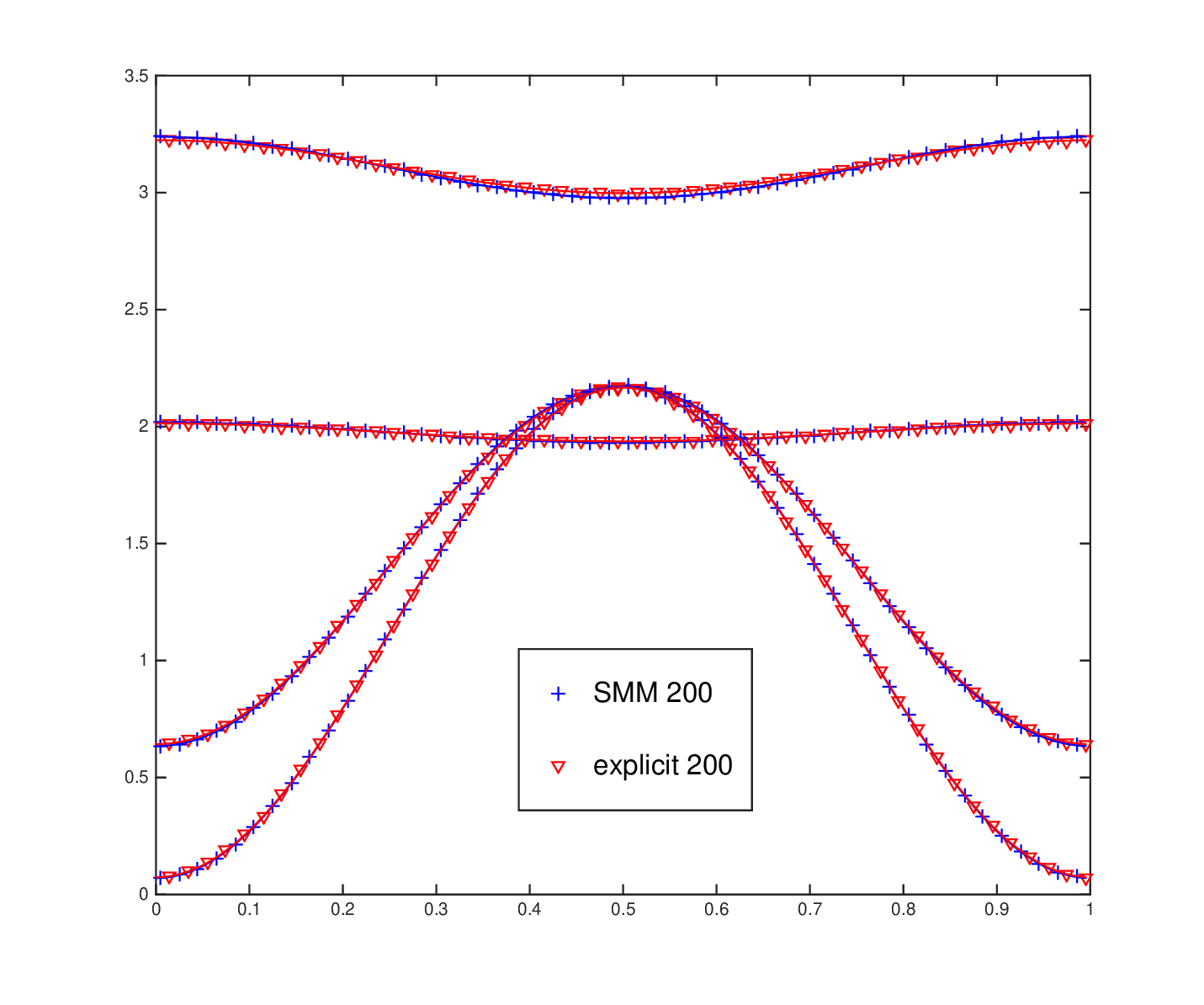}
 \end{center}
  \caption{one-group transport equation $\varepsilon = 1$ : summary of the previous comparisons between SMM and explicit schemes (200 grid points): Results at times t = 0.1, 0.3, 0.6 and 1 (from down to up) .\label{c}}
 \end{figure}
 
 \newpage
 \label{numerics}
 \begin{figure}
 \begin{center}
  \includegraphics[scale=0.57]{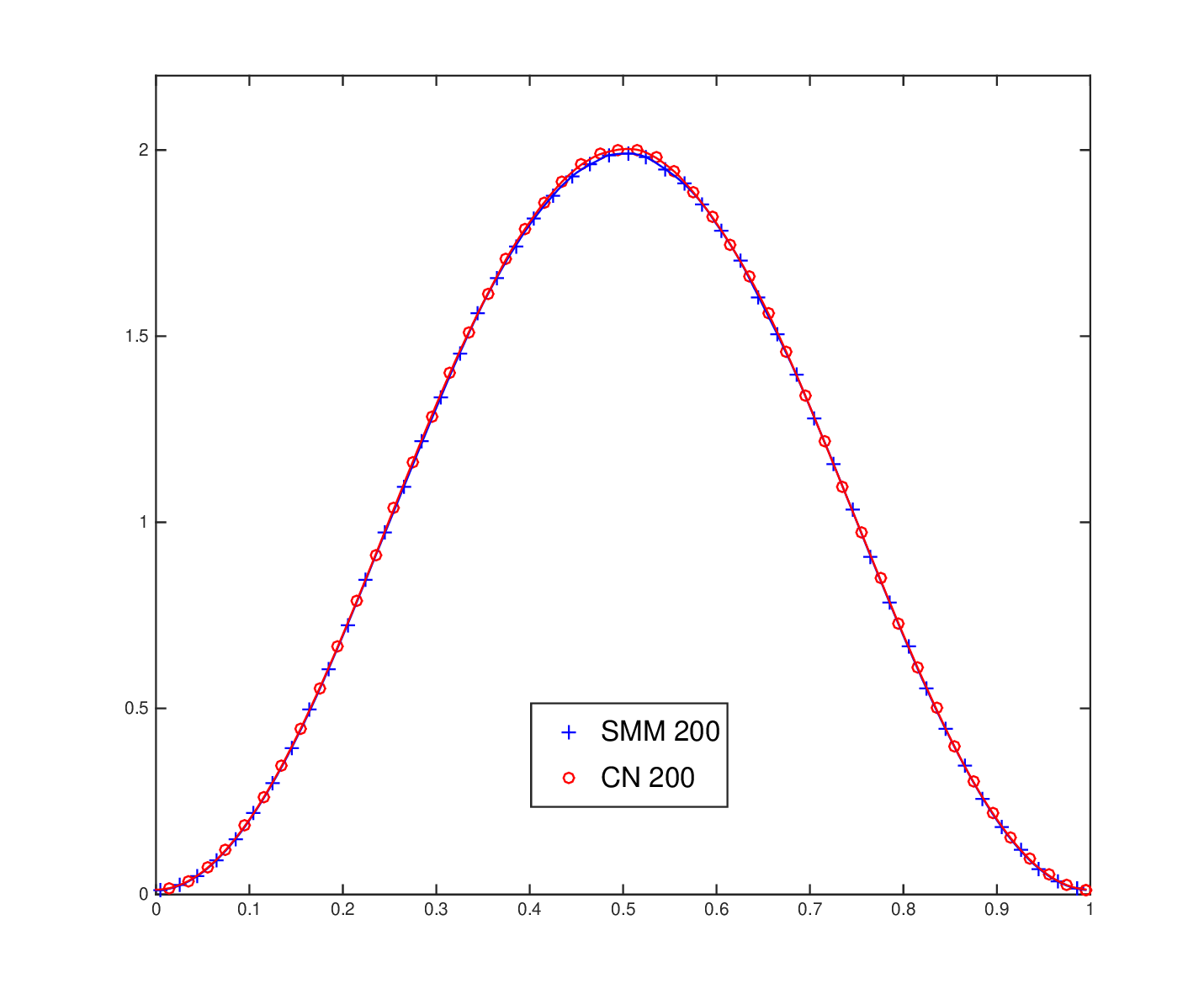}
  \includegraphics[scale=0.57]{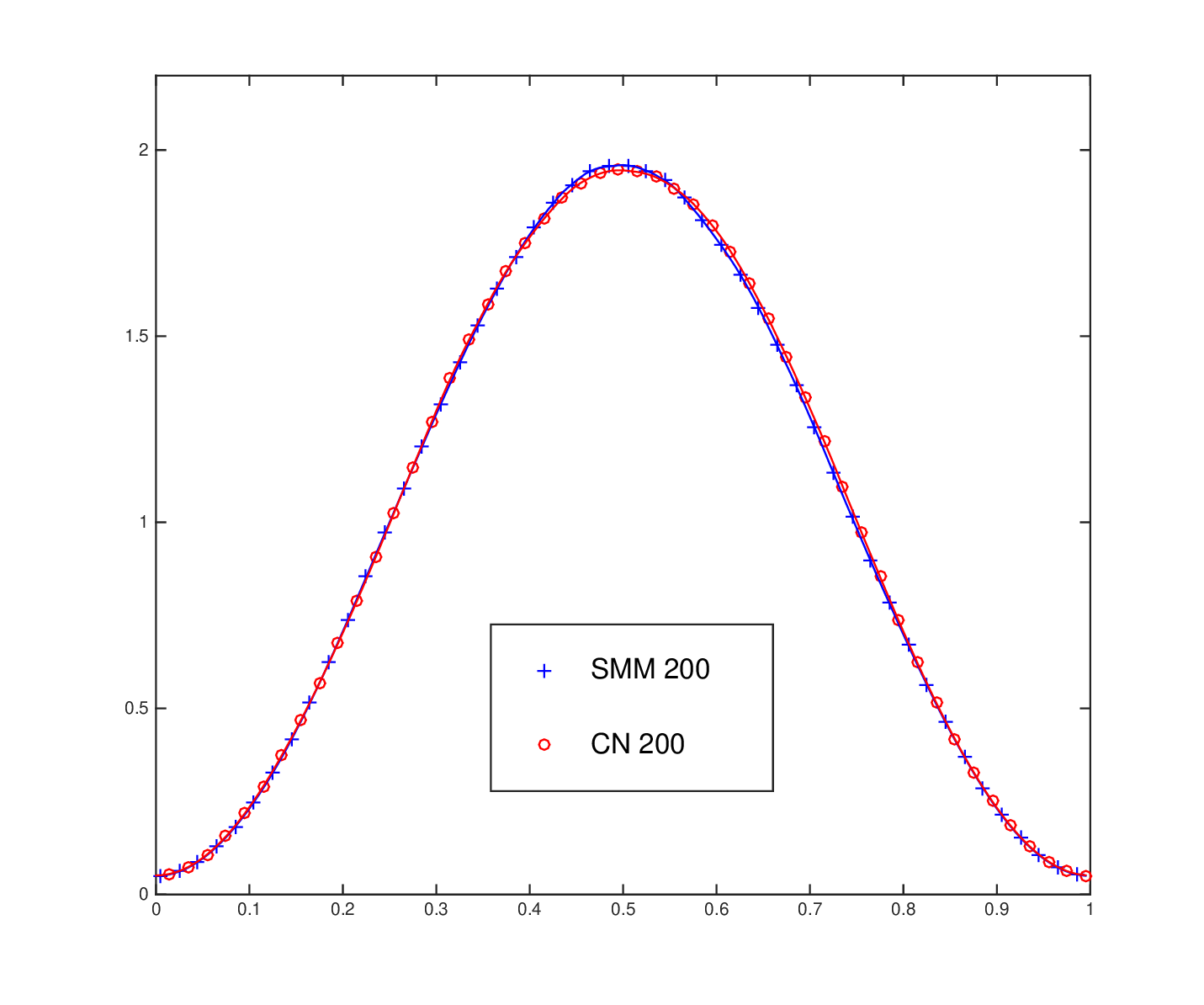}
 \end{center}
  \caption{one-group transport equation $\varepsilon = 10^{-2}$ : comparison between SMM and CN schemes (200 grid points): $t = \varepsilon /10$ (top), $t = 4 \varepsilon /10$ (bottom).\label{d1}}
 \end{figure}
 
  \newpage
 \label{numerics}
 \begin{figure}
 \begin{center}
 \includegraphics[scale=0.57]{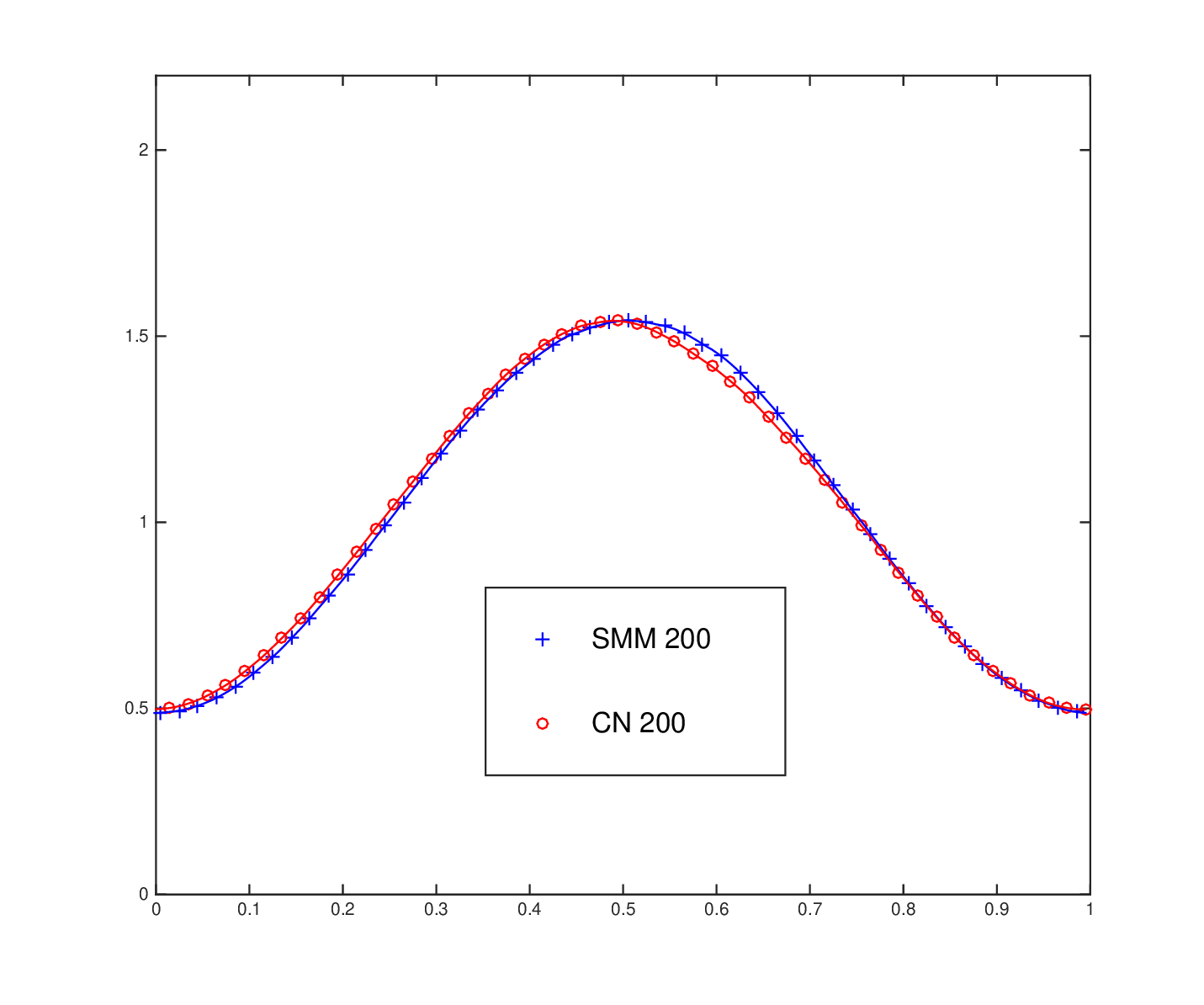}
  \includegraphics[scale=0.57]{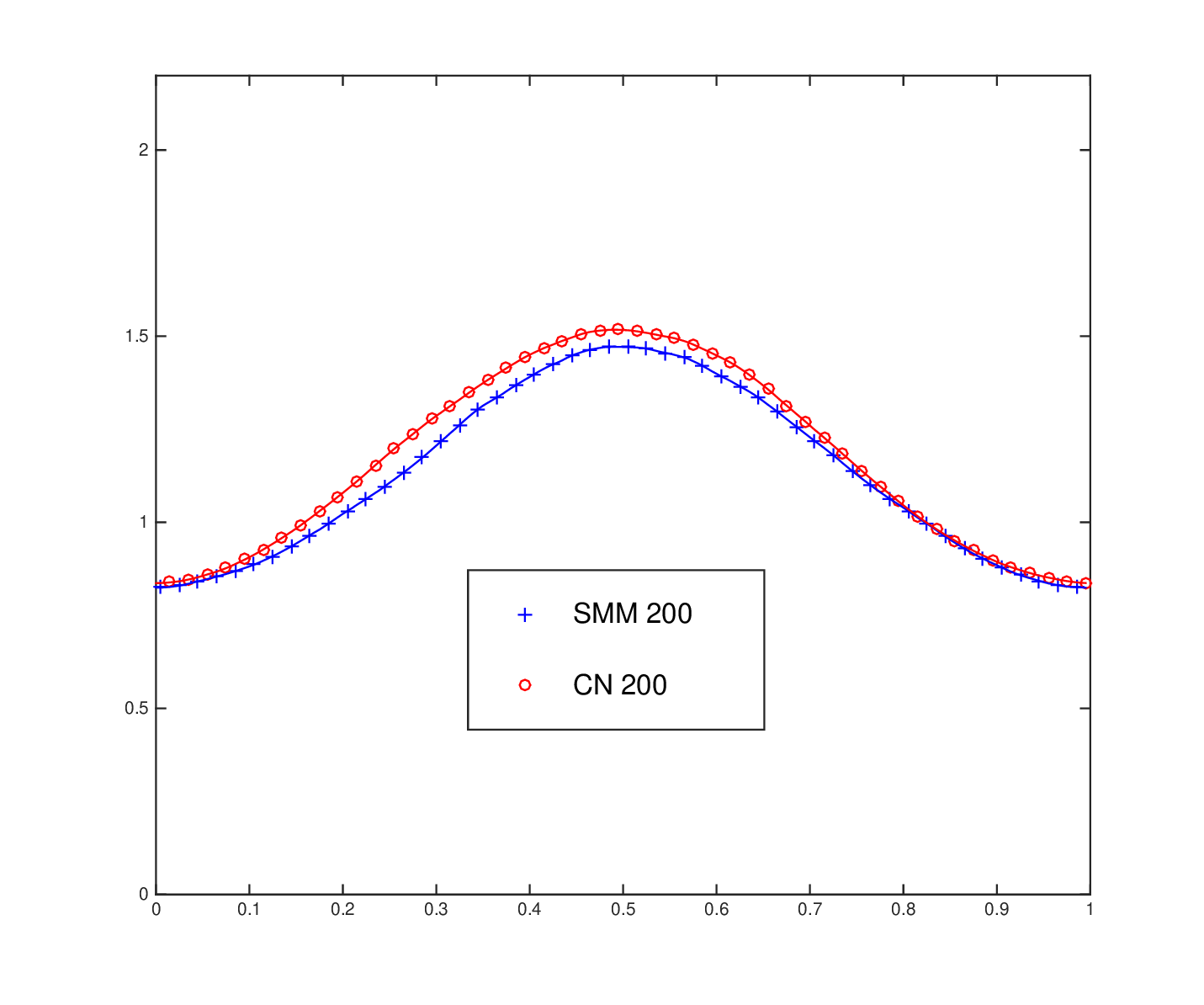}
 \end{center}
  \caption{one-group transport equation $\varepsilon = 10^{-2}$ : comparison between SMM and CN schemes (200 grid points): $t = 0.05$ (top), $t = 0.1$ (bottom).\label{d2}}
 \end{figure}

\end{document}